\documentclass[10pt]{amsart}
\pdfoutput=1

\usepackage{amsmath}
\usepackage{amssymb}
\usepackage{amsthm}
\usepackage{mathtools}
\usepackage{verbatim}
\usepackage{float}
\usepackage{enumerate}
\usepackage{bm}
\usepackage{algorithm}
\usepackage{algpseudocode}
\usepackage[export]{adjustbox}

\usepackage{tikz}
\usepackage{subcaption}
\usetikzlibrary{hobby}
\usetikzlibrary{fadings}
\usetikzlibrary{positioning}
\usetikzlibrary{arrows.meta, bending, decorations.markings}
\tikzset{
    base/.style = {rectangle, rounded corners, draw=black, minimum width=4cm, minimum height=1cm, text centered, font=\sffamily},
    box/.style = {rectangle, rounded corners, minimum width=5cm, minimum height=2cm, text centered, draw=gray, fill=black!15},
    header/.style={
    label={[rectangle, fill=white, draw, anchor=center, minimum width=2cm, node font=\ttfamily, name=\tikzlastnode-header]north:{#1}}}}
\tikzstyle{arrow} = [very thick, ->, >=stealth]

\usepackage{shortcuts}

\begin{document}

\author{Linhang Huang}
\email{lhhuang@uw.edu}
\pagestyle{plain}
\title{Electrostatic Skeletons and\\ Condition of Strict Descent}

\begin{abstract}
Given a precompact domain $\Omega \subseteq \R^2$, the electrostatic skeleton of $\Omega$ is defined as a positive measure inside $\Omega$, supported on a set with no simple loops, which generates $\partial \Omega$ as an equipotential curve. Eremenko conjectured that every convex polygon admits a unique electrostatic skeleton. This conjecture has since been proven for triangles and regular polygons. In this paper, we will prove the conjecture for quadrilaterals with a line of symmetry using arguments from conformal geometry. We will also discuss a natural condition that implies the existence of electrostatic skeletons.
\end{abstract}

\maketitle

\vspace{-0pt}

\section{Introduction and Results}
\subsection{Introduction}
Given a finite, compactly supported positive Borel measure $\mu$, the \textit{logarithmic potential} of $\mu$ is defined as $U^\mu: \C \to (-\infty,\infty]$ where
\begin{equation*}
    U^\mu (z) = \int \log \frac{1}{|z-w|} ~d\mu(w).
\end{equation*} 
   
An electrostatic skeleton is a positive measure supported on a tree inside the polygon that recovers the Green's function with its potential and, in particular, it produces the polygon as a level set of its potential. A formal definition is as follows. 
 
\begin{quote} \textbf{Definition.} A  positive measure $\mu$ is the \textit{electrostatic skeleton} for a precompact domain $\Omega$ if $\mu$ generates the same logarithmic potential as the equilibrium measure $\mu_E$ outside $\Omega$ and $\supp \mu$ is inside $\overline \Omega$ and contains no simple loops.
\end{quote}

\begin{figure}[h!]
    \begin{subfigure}{.45\textwidth}
      \centering
      \includegraphics[width=0.9\linewidth]{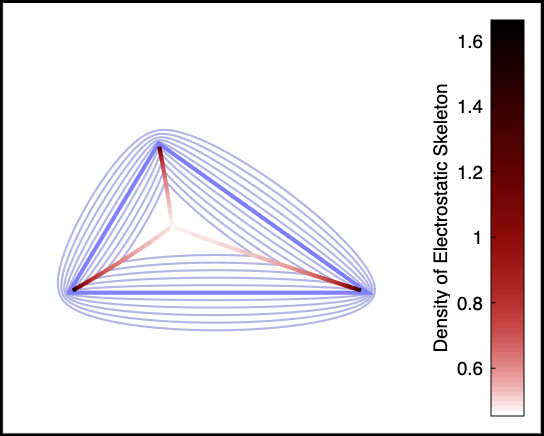}
    \end{subfigure}%
    \begin{subfigure}{.45\textwidth}
      \centering
      \includegraphics[width=0.9\linewidth]{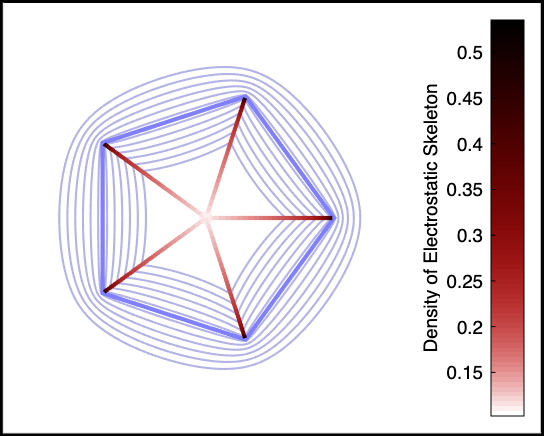}
    \end{subfigure}
    \caption{Electrostatic skeletons for a triangle and the regular pentagon, with level sets of the logarithmic potentials in blue.}
    \label{fig:tri_penta}
\end{figure}

Note that the condition $U^{\mu}|_{\Omega^C} = U^{\mu_E}|_{\Omega^C}$ is equivalent to requiring that $-U^\mu - \log\mathrm{Cap}(\Omega)$ is the Green's function of $\Omega$ on $\C \backslash \Omega$. The question of understanding which sets have an electrostatic skeleton was first raised by E.B. Saff in 2003 \cite{Lundberg2013lemniscate}. Saff asked whether every convex polygon admits an electrostatic skeleton. Since then,  Lundberg and Ramachandran \cite{Lundberg2015electrostatic} proved the existence and uniqueness of electrostatic skeleton for triangles (and also simplices in higher dimensional generalization) using the subharmonic extension of the Green's function (see also \cite{eremenko2013electrostatic}).  The existence result for regular polygons was proven by Lundberg and Totik \cite{Lundberg2013lemniscate}, who also mentioned its relation to the concept of \textit{lemniscate growth}. In particular, they show that if the support of the skeleton does not intersect the boundary of the domain, then domain can be shrunk in a way the preserve the Green's function. In a sense, the electrostatic skeleton outlines the maximum domain of conformal invariance. In an example of this phenomenon is for the ellipse which can be shrunk to the linear segment between foci. The example skeletons for a triangle and a regular pentagon are plotted in Figure~\ref{fig:tri_penta}.\\

The task of finding the electrostatic skeleton is essentially the one of finding fictitious charges on a tree-like subset inside the domain which can replace the equilibrium charges. This relates to the classic technique in the study of electrostatics called the \textit{method of image charges} \cite{Jeans_2009}. In the context of potential theory, the measure for the electrostatic skeleton would generate the equilibrium measure as the \textit{balayage} on the boundary of the domain. A similar question with potentials generated by Lebesgue measure was also considered by Gustafsson in \cite{Gustafsson1998mother}. Electrostatic skeletons have also been shown by Díaz, Saff and Stylianopoulos \cite{mina2005zero} to relate to Bergman orthogonal polynomials.

\begin{figure}[h!]
    \begin{subfigure}{.5\textwidth}
      \centering
      \includegraphics[width=0.95\linewidth]{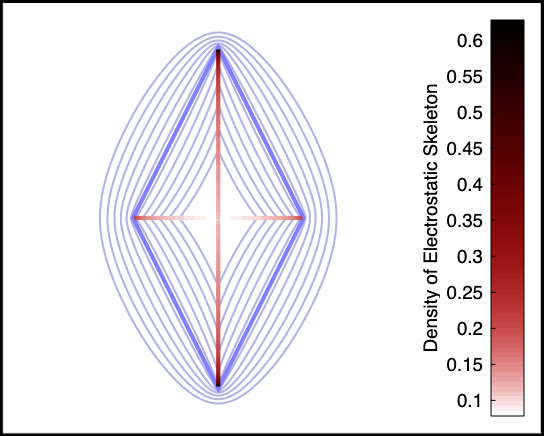}
    \end{subfigure}%
    \begin{subfigure}{.5\textwidth}
      \centering
      \includegraphics[width=0.95\linewidth]{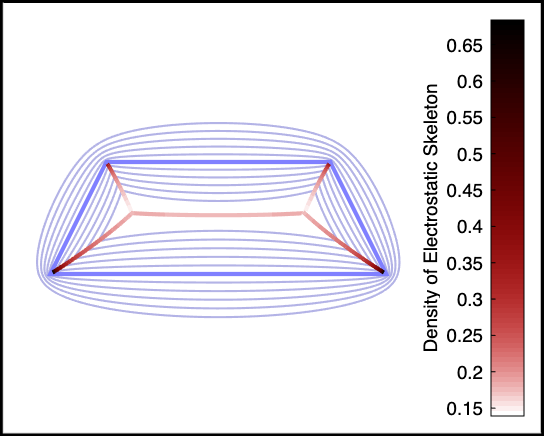}
    \end{subfigure}
    \caption{Electrostatic skeletons for a kite shape and an isosceles trapezoid, with level sets of the logarithmic potentials in blue}
    \label{fig:kite_trape}
\end{figure}

\subsection{Geometric Results}
We prove the existence of electrostatic skeletons for symmetric convex quadrilaterals: we will first show that the convex kite shapes (convex quadrilaterals symmetric with respect to one of their diagonals)  admit electrostatic skeletons (see Fig.~\ref{fig:kite_trape}).

\begin{theorem}\label{theo_kite} Each convex kite shape admits an electrostatic skeleton.
\end{theorem}

Our second result deals with a second family of quadrilaterals: we consider the isosceles trapezoid where the line of symmetry does not pass through any vertex and show that all these shapes admit electrostatic skeletons.

\begin{theorem} \label{theo_trap} Each isosceles trapezoid admits an electrostatic skeleton.
\end{theorem}
For proofs of both Theorems~\ref{theo_kite} and \ref{theo_trap}, we leverage heavily with the line of symmetry. Thus it is hard to generalize the methods for a generic convex polygon. 

\subsection{Structural Results} \label{notation}
We now discuss the main result: all convex polygons that satisfy what we call the "strict descent" condition admit electrostatic skeletons. First note that the Green's function $g$ of a polygon $\Omega$ is a harmonic function defined on the outside of the polygon. The defining property of $g$ is that it is a non-negative function with logarithmic growth outside the polygon and vanishes on the edges of the polygon. Given an edge $[a,b]$ of the polygon, we can reflect the Green function using Schwarz reflection with respect to each side by setting \begin{equation*}
    \tilde{g}(z) = -g\left(a + \frac{b-a}{\overline{b-a}}(\overline{z-a})\right).
\end{equation*}

Given a convex polygon $\Omega$, we denote its vertices as $z_1,\dots,z_n \in \mathbb{C}$ and edges which are finite line segments $L_1,\dots,L_n$ (see Figure~\ref{fig:show}) such that \begin{equation*}
    L_i \cap L_{i+1} = \set{z_i}. 
\end{equation*} We shall denote the Schwarz reflection of $g$ with respect to $L_i$ by $g_i$. We denote the line containing the line segment $L_i$ by $\ell_i$. Note that given each pair of $g_i$ and $g_j$, we can extend both functions to the component of $\C \backslash (\ell_i \cup \ell_j)$ that contains $\Omega$. We denote the closure of this component (in $\C$) by $W_{i,j}$. Note that both $g_i$ and $g_j$ can be further extended continuously to $\partial W_{i,j}$ (on $\ell_i$ for instance, we have $g_i(z) = -g(z)$). We set \begin{equation*}
    S_{i,j} := \set{z\in W_{i,j}: g_i(z) = g_j(z)}.
\end{equation*}

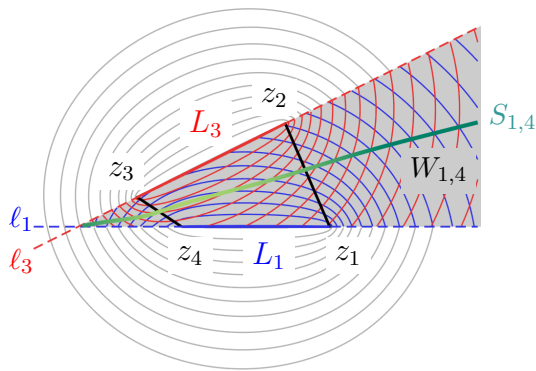
\begin{figure}[h!]
    \centering
    \begin{tikzpicture}[scale=0.7]
        \node[inner sep=0pt] (julia) at (0,0){\includegraphics[width=0.5\linewidth]{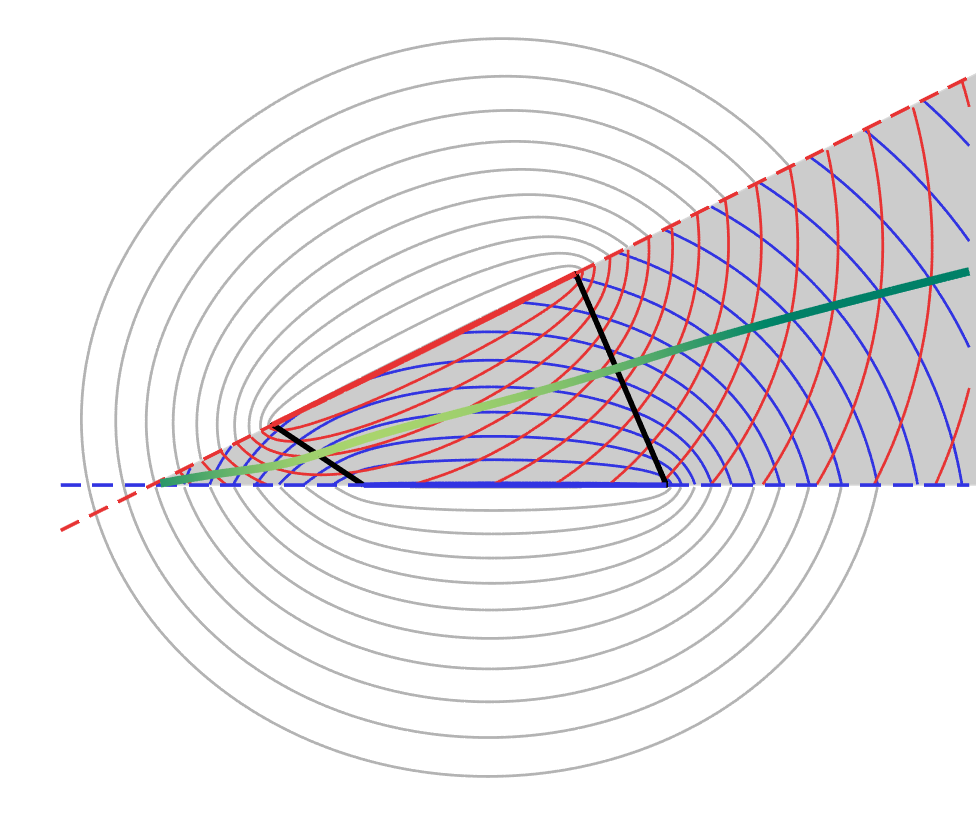}};
        \node[red,fill=white] at (-0.7,1.3) {\Large {$L_3$}};
        \node[blue,fill=white] at (0.5,-1.2) {\Large {$L_1$}};
        \node[fill=white] at (-1,-1.2) {\Large $z_4$};
        \node[fill=white] at (2,-1.2) {\Large $z_1$};
        \node[fill=white] at (0.6,1.8) {\Large $z_2$};
        \node[fill=white] at (-2.3,0.4) {\Large $z_3$};
        \node[right, teal!80!white] at (4.5,1.4) {\large $S_{1,4}$};
        \node[right, fill=black!20!white] at (3,0.4) {\large $W_{1,4}$};
        \node[red] at (-4.2,-1.3) {\large $\ell_3$};
        \node[blue] at (-4.2,-0.5) {\large $\ell_1$};
    \end{tikzpicture}
    \caption{Level sets of $g_1$ and $g_3$ are highlighted in red and blue respectively. The wedge $W_{1,4}$ is gray area between $\ell_1$ and $\ell_3$ that contains the polygon.}
    \label{fig:show}
\end{figure}

\begin{definition}[Condition of Strict Descent]\label{con_sd1}
   We say a convex polygon $\Omega$ satisfies \textit{the strict descent condition} if  given any two distinct Schwarz reflections $g_i$ and $g_j$, for any $z\in S_{i,j}$ such that $\nabla g_i(z)$ is parallel to $\nabla g_j(z)$, we have $\nabla g_i(z) \cdot \nabla g_j(z) < 0$. 
\end{definition}

The motivation for the strict descent condition is that we want to construct the electrostatic skeleton on where different Schwarz reflections coincide. Loosely speaking, this condition implies that we can build the skeleton by looking at level sets of the Schwarz reflection with increasingly smaller value. 

\begin{theorem}[Main Result] \label{theo_descent}
    Each polygon that satisfies the strict descent condition admits an electrostatic skeleton. Moreover, the support of the skeleton is piecewise analytic and consists of at most $2n-3$ analytic curves and the measure is equivalent to $\mathcal{H}^1$ on the support.
\end{theorem}

Theorem~\ref{theo_descent} implies that the existence of electrostatic skeleton can be reduced to a somewhat explicit analytic property of the Green's function of the polygon. The bound $2n-3 = n + (n-3)$ comes from a correspondence between the shape of the skeleton and a partition of $n$-gon. Thus the number of arcs that are not rooted at the vertices is bounded by the number of non-crossing matching (see Figure~\ref{fig:hepta}).

We argue that the strict descent condition is quite natural and should hold for generic polygons. The author has not found any convex polygon that does not satisfy the strict descent condition. This suggests the following strenghtening of the existing conjecture (see above).

\begin{conjecture} \label{conj_descent}
    All convex polygons satisfy the strict descent condition.
\end{conjecture}

Whether or not a convex polygon satisfies the strict descent condition is a property of its Riemann map: as such, it can be explicitly computed for any arbitrary example. While, in principle, being completely explicit (using only the exterior Schwarz-Christoffel mapping and Schwarz reflection), we have been unable to deal with the problem at a general level.

\begin{figure}[h!]
    \begin{subfigure}{.55\textwidth}
      \centering
      \begin{tikzpicture}[scale=1]
          \node at (0,0) {\includegraphics[width=0.95\linewidth]{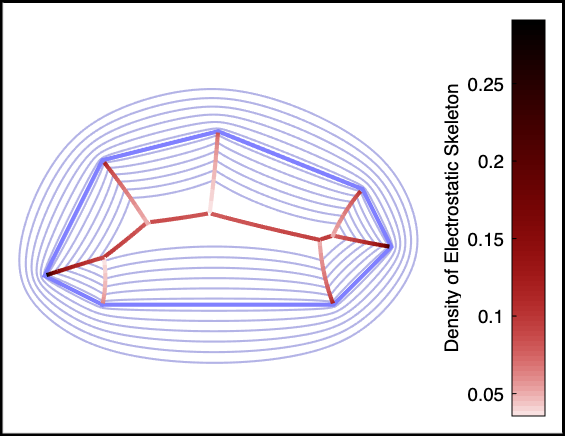}};
          \node[fill=white] at (-0.8,-1.4) {$L_1$};
          \node[fill=white] at (1.3,-1) {$L_2$};
          \node[fill=white] at (1.5,0.3) {$L_3$};
          \node[fill=white] at (0.3,1.1) {$L_4$};
          \node[fill=white] at (-1.5,1.3) {$L_5$};
          \node[fill=white] at (-2.8,0.4) {$L_6$};
          \node[fill=white] at (-2.7,-1.3) {$L_7$};
      \end{tikzpicture}
      
    \end{subfigure}
    \begin{subfigure}{.4\textwidth}
    \centering
      \begin{tikzpicture}[line width = 1.5]        
      \draw[red] ({2*cos(360/7)},{2*sin(360/7)}) -- ({2*cos(360/7*6)},{2*sin(360/7*6)});
        \draw[red] ({2*cos(360/7)},{2*sin(360/7)}) -- ({2*cos(360/7*5)},{2*sin(360/7*5)});
        \draw[red] ({2*cos(360/7*2)},{2*sin(360/7*2)}) -- ({2*cos(360/7*5)},{2*sin(360/7*5)});
        \draw[red] ({2*cos(360/7*3)},{2*sin(360/7*3)}) -- ({2*cos(360/7*5)},{2*sin(360/7*5)});
          \draw ({2*cos(0)},{2*sin(0)}) -- ({2*cos(360/7)},{2*sin(360/7)}) --
          ({2*cos(360/7*2)},{2*sin(360/7*2)}) --
          ({2*cos(360/7*3)},{2*sin(360/7*3)}) --
          ({2*cos(360/7*4)},{2*sin(360/7*4)}) --
          ({2*cos(360/7*5)},{2*sin(360/7*5)}) --
        ({2*cos(360/7*6)},{2*sin(360/7*6)})  -- cycle;
        \node[right] at  ({2*cos(0)},{2*sin(0)}) {$3$};
        \node[above right] at  ({2*cos(360/7)},{2*sin(360/7)}) {$4$};
        \node[above] at  ({2*cos(360/7*2)},{2*sin(360/7*2)}) {$5$};
        \node[left] at  ({2*cos(360/7*3)},{2*sin(360/7*3)}) {$6$};
        \node[left] at  ({2*cos(360/7*4)},{2*sin(360/7*4)}) {$7$};
        \node[below] at  ({2*cos(360/7*5)},{2*sin(360/7*5)}) {$1$};
        \node[below] at  ({2*cos(360/7*6)},{2*sin(360/7*6)}) {$2$};
      \end{tikzpicture}
    \end{subfigure}
    \caption{Electrostatic skeleton for a heptagon and its corresponding partition of (regular) heptagon.}
    \label{fig:hepta}
\end{figure}
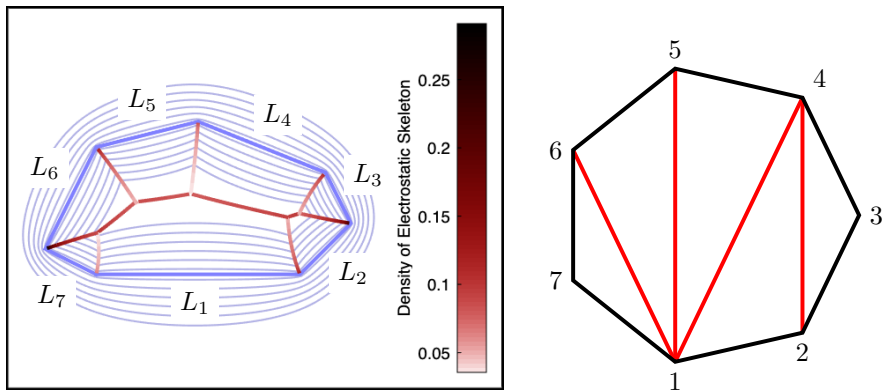

\section{Preliminary}
\subsection{Logarithmic Potential and Green's Function}
We review some basic properties of logarithmic potentials that will be used for the proofs. 
We refer to Saff's survey \cite{Saff2010logarithmic} and the book of Saff and Totik \cite{saff1997logarithmic} for further details.

\begin{lemma}[Riesz decomposition, see {\cite[Theorem 2.4]{Saff2010logarithmic}}]\label{log}
    Given a finite, compactly supported Borel measure $\mu$,
    \begin{enumerate}[(a)]
        \item $U^\mu$ is superharmonic. In particular, it is lower-semicontinuous,
        \item $U^\mu$ is finite and harmonic on $\C \backslash \supp \mu$.
    \end{enumerate}
    Conversely, given a superharmonic function $f:\C \to (-\infty ,\infty]$ which is harmonic outside a compact set $K$, there exists a unique positive Borel measure \begin{equation*}
        \mu = -(2\pi)^{-1}\Delta f,
    \end{equation*}
    which supported on $K$ such that $U^\mu - f$ is constant. 
\end{lemma}

Here the Laplacian is defined in the distributional sense as in \cite[Section~3.7]{ransford1995potential}. Given a compact set $D \subseteq \C$ with positive capacity, the Green's function $g:\C \backslash \Int D \to \R$ is the unique harmonic function that vanishes on $\partial D$ and satisfies $g(z) = \log |z| +O(1)$. When $\hat\C \backslash D$ is simply-connected, we have $g(z) = \log|f(z)|$, where $f: \hat\C \backslash D \to \hat \C\backslash \D$ is the unique Riemann mapping function satisfying $f(\infty) = \infty$ and $f'(\infty) > 0$.  

\begin{lemma}[Equipotential Lines for Convex Domains {\cite[Theorem~2.9]{pommerenke1975univalent}}]\label{convex} Let $D$ be a compact convex set and $f:\hat \C \backslash D \to \hat \C \backslash \overline \D$ be the biholomorphic mapping such that $f(\infty) = \infty$ and $f'(\infty) >0$. Then for any $r>1$, the set $\set{z:|f(z)|=r}$, which is the equipotential line of the Green's function, has strictly convex interior.
\end{lemma}

The task of finding an electrostatic skeleton is related to finding a subharmonic extension of $g$ inside the original domain $\Omega$, as indicated by the following Lemma.

\begin{lemma} \label{subhar} 
    Given a Jordan domain $\Omega$ with its Green's function $g:\C\backslash \Omega\to \R$, suppose $\tilde g$ is a subharmonic function over $\C$ such that $\tilde g|_{\C \backslash \Omega} = g$. Then we have $g(z) = -U^\mu(z) - \log \mathrm{Cap}(\Omega)$ for $z\in \C \backslash \Omega$, where $\mu = (2\pi)^{-1}\Delta \tilde g$.
\end{lemma}

\begin{proof}
    Let $g_0$ be the extension of $g$ such that $g_0(z) =0$ for $z\in \Omega$. Then we know that the equilibrium measure $\mu_E = (2\pi)^{-1}\Delta g_0$ satisfies that $g = -U^{\mu_E} -\log \mathrm{Cap}(\Omega)$ outside $\Omega$. So it suffices to prove that $U^{\mu_E} = U^\mu$ for any subharmonic extension $\tilde g$ of $g$ and $\mu = (2\pi)^{-1}\Delta \tilde{g}$. Given any $C^{2}$ Jordan domain $\mathcal{O}$ containing $\overline{\Omega}$ and any $z$ outside $\mathcal{O}$, the function $w\mapsto -\log|z-w|$ is harmonic inside $\mathcal{O}$. By the Gauss-Green formula, we have \begin{equation*}
        U^\mu (z) = \int \log \frac{1}{|z-w|}d\mu(w) = -\int_{\partial \mathcal{O}} \log\frac{1}{|z-\zeta|} \pder {\tilde{g}}{\mathbf{n}}(\zeta) |d\zeta|.
    \end{equation*} However, $\tilde g$ equals $g$ on $\partial \mathcal{O}$. Thus the value of $U^\mu$ outside $\mathcal{O}$ is the same as $U^{\mu_E}$. Since $\mathcal{O}$ is arbitrary, the result then follows.
\end{proof}

\subsection{Geometry of the zero sets}
Recall the notations established in Section~\ref{notation}.

\begin{lemma}\label{no_branch}
    $S_{i,j}$ consists of a single analytic curve. If $\ell_i \cap \ell_j = \set{z_{i,j}}$, then $S_{i,j}$ intersects $\partial W_{i,j}$ only at $z_{i,j}$. If $\ell_i$ and $\ell_j$ are parallel (intersect at $\infty$), then $S_{i,j}$ is disjoint from $\partial W_{i,j}$ and it intersects every line perpendicular to $\ell_i$.
\end{lemma}

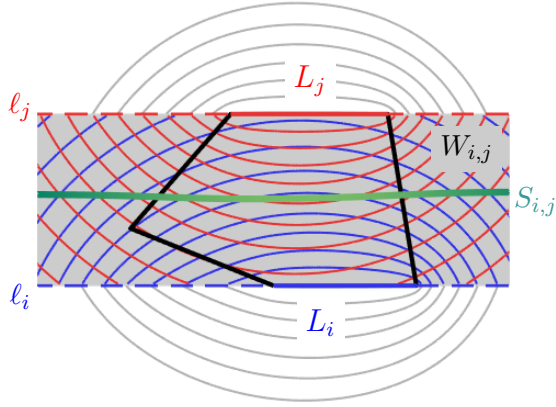
\begin{figure}[h!]
    \centering
        \begin{tikzpicture}[scale=0.7]
        \node[inner sep=0pt] (julia) at (0,0){\includegraphics[width=0.5\linewidth]{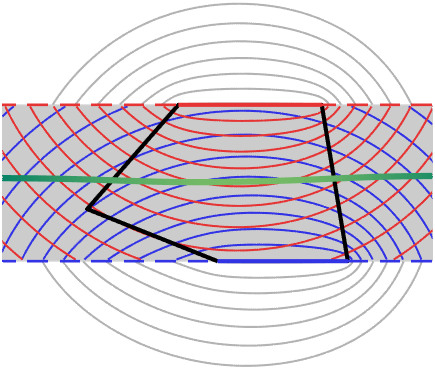}};
        \node[red,fill=white] at (0.7,2.3) {\Large {$L_j$}};
        \node[blue,fill=white] at (0.9,-2.3) {\Large {$L_i$}};
        \node[right, teal!80!white] at (4.4,0) {\large $S_{i,j}$};
        \node[right, fill=black!20!white] at (3,1) {\large $W_{i,j}$};
        \node[red] at (-4.8,1.8) {\large $\ell_j$};
        \node[blue] at (-4.8,-1.8) {\large $\ell_i$};
    \end{tikzpicture}
    \caption{The case where $\ell_i$ and $\ell_j$ are parallel}
    \label{fig:parallel}
\end{figure}

\begin{proof}
    Note that in the case when $\ell_i$ and $\ell_j$ intersects at $z_{i,j}$, we have \begin{equation*}
        -g_i(z_{i,j}) = -g_j(z_{i,j}) = g(z_{i,j}).
    \end{equation*} Thus we have $z_{i,j} \in S_{i,j}$. Given any $z\in (\ell_i \cap W_{i,j})\backslash \set{z_{i,j}}$, we set $z'$ to be the reflection of $z$ with respect to $\ell_j$ (see Figure~\ref{fig:reflection}). We then have that \begin{equation*}
        -g_i(z) = g(z),\quad -g_j(z) = g(z').
    \end{equation*}
    Since the support of the equilibrium $\mu_E$ (which is $\partial \Omega$) lies on $z$'s side of $\C \backslash\ell_j$, every point on $\partial\Omega$ is closer to $z$ than $z'$. Let $\mu_E$ be the equilibrium measure of $\Omega$. We thus have \begin{align*}
        g(z) &= -\log \mathrm{Cap}(\Omega) -\int \log \frac{1}{|z-w|}d\mu_E(w) \\&<-\log \mathrm{Cap}(\Omega) -\int \log \frac{1}{|z'-w|}d\mu_E(w) \\
        &=g(z').
    \end{align*} This implies that $z\notin S_{i,j}$. Hence, $S_{i,j}$ does not intersect $(\ell_i \cap W_{i,j})$ except possibly at $z_{i,j}$. Similarly, it can be shown that $S_{i,j}$ does not intersect $(\ell_j\cap W_{i,j})$, except possibly at $z_{i,j}$.

    \begin{figure}[h!]
        \centering
        \begin{tikzpicture}[scale = 2]
            \draw[red, line width=1.5] (1.5, 0) node[right]{\large $\ell_j$} -- (0,0) -- (-1,0.8) node[above left=0.3pt] {\large $\ell_i$};
            \draw[red, dashed, line width=1.5] (0,0) -- (-1,-0.8);
            \fill (-0.8, 0.64) circle (1pt) node[above right=0.3pt] {\large $z$};
            \fill (-0.8, -0.64) circle (1pt) node[below right=0.3pt] {\large $z'$};
            \draw[blue, line width=1.6] (0.8,0) -- (0.3,0) -- (-0.2,0.16) -- (-0.7,0.56) -- (1.3,1) --cycle;
            \node[blue] at (0.35,0.4) {\huge $\Omega$};
            \node[blue, left] at (-0.4, 0.25) {\large $L_i$};
            \node[blue, below] at (0.55, 0) {\large $L_j$};
        \end{tikzpicture}
        \caption{Idea behind the proof of Lemma \ref{no_branch}.}
        \label{fig:reflection}
    \end{figure}
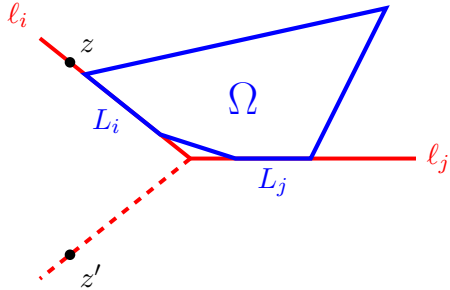

    Now we prove the rest of the assertions. Take any  component $C$ of $W_{i,j}\backslash S_{i,j}$, if $C \cap \partial W_{i,j} = \emptyset$, then we have $\partial C \subseteq S_{i,j}$. Since $S_{i,j}$ is the zero set of a harmonic function, we have that $C$ cannot be bounded. Without loss of generality, we assume that $g_i-g_j$ is positive in $C$. Then the function \begin{equation*}
        u(z) := \begin{cases}
            g_i(z) - g_j(z), &z\in C,\\
            0, &z\in W_{i,j}\backslash C.
        \end{cases}
    \end{equation*}
    Since $u$ is harmonic on $W_{i,j} \backslash \partial C$ and satisfies the mean value inequality on $\partial C$ for sufficiently small $r>0$: \begin{equation*}
        u(z) = 0 \leq \frac{1}{2\pi}\int^{2\pi}_0 u(z+re^{i\theta})dr.
    \end{equation*}
    We have that $u$ is a subharmonic function vanishes on the sector boundary $\partial W_{i,j}$ and satisfies the logarithmic growth condition \begin{equation*}
        u(z) = O(\log|z|).
    \end{equation*}
    
    By the Phragmén-Lindelöf theorem for subharmonic functions (see \cite{Telyakovski1999phragmen}), we have that $u\equiv 0$ which would imply $g_i \equiv g_j$. But this contradicts the fact that $S_{i,j} \neq W_{i,j}$. Therefore, all the components of $W_{i,j}\backslash S_{i,j}$ intersect either $\ell_i$ or $\ell_j$ and are unbounded. Since $S_{i,j}$ is disjoint from $\ell_i$ and $\ell_j$ except for the two lines' intersection, we have that there is only one component of $W_{i,j}\backslash S_{i,j}$ that intersects $\ell_i$ and one intersecting $\ell_j$. We denote the components of $W_{i,j}\backslash S_{i,j}$ that contains $L_i$ by $C_i$ and define $C_j$ similarly. It follows that $C_i$ and $C_j$ are the only two components and we have that $C_i \neq C_j$. The result then follows.
\end{proof}

\section{Electrostatic Skeleton for Quadrilaterals}\label{quad}
\subsection{Proof of Theorem \ref{theo_kite}} Theorem~\ref{theo_kite} can be considered as an immediate consequence of Lemma~\ref{no_branch}, as the following proof illustrates. Note that a kite shape is a quadrilateral symmetric with respect to one of its diagonals.

\begin{proof} 
    Given a convex kite shape $\Omega$, without loss of generality, we assume the line of symmetry $\ell$ passes $z_1$ and $z_3$. By symmetry, we then have that both $S_{1,2}$ and $S_{3,4}$ are contained in $\ell$. Lemma~\ref{no_branch} implies that there exists a unique (closed) segment of $S_{2,3}$ that connects $z_2$ to $\ell$. 
    We shall denote this segment by $K_1$ and denote the point when $K_1$ and $\ell$ intersect by $c$. By symmetry, there exists a unique segment $K'_1$ of $S_{1,4}$ which connects $z_4$ with $\ell$ and also reaches $\ell$ at $c$. The segment $K'_1$ is the reflection of $K_1$ with respect to $\ell$.
    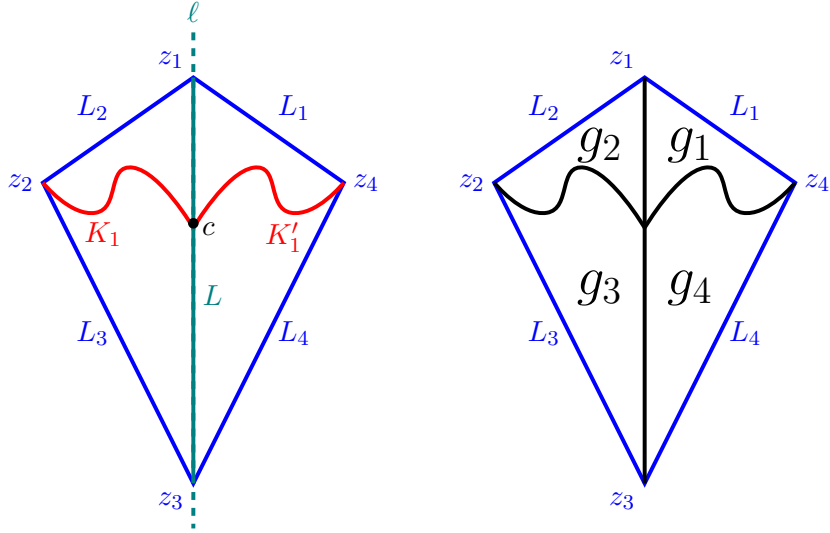
\begin{figure}[h!]
        \centering
        \begin{tikzpicture}[scale = 2]
        \begin{scope}
            \draw[dashed, line width=1.5, teal] (0,1) node[above] {\large $\ell$}-- (0,-2.3);
            \draw[blue, line width=1.5] (1,0) -- (0, 0.7) -- (-1,0) -- (0,-2) -- cycle;
            \draw[red, line width=1.5] plot[smooth, tension=1] coordinates {(-1,0) (-0.65,-0.2) (-0.4,0.1) (0,-0.3)};
            \draw[red, line width=1.5] plot[smooth, tension=1] coordinates {(1,0) (0.65,-0.2) (0.4,0.1) (0,-0.3)};
            \draw[teal, line width=1.5] (0,0.7) -- (0,-2);
            \node[above left, blue] at (0, 0.7) {\large $z_1$};
            \node[below left, blue] at (0,-2) {\large $z_3$};
            \node[left, blue] at (-1, 0) {\large $z_2$};
            \node[right, blue] at (1, 0) {\large $z_4$};
            \node[left, blue] at (-0.5,0.5) {\large $L_2$};
            \node[right, blue] at (0.5,0.5) {\large $L_1$};
            \node[left, blue] at (-0.5,-1) {\large $L_3$};
            \node[right, blue] at (0.5,-1) {\large $L_4$};
            \node[right, teal] at (0,-0.75) {\large $L$};
            \node[below, red] at (-0.6,-0.2) {\large $K_1$};
            \node[below, red] at (0.6,-0.2) {\large $K_1'$};
            \fill[black] (0,-0.27) circle (1pt);
            \node[below] at (0.1,-0.2) {\large $c$};
        \end{scope}

        \begin{scope}[shift = {(3,0)}]
            \draw[blue, line width=1.5] (1,0) -- (0, 0.7) -- (-1,0) -- (0,-2) -- cycle;
            \draw[line width=1.5] plot[smooth, tension=1] coordinates {(-1,0) (-0.65,-0.2) (-0.4,0.1) (0,-0.3)};
            \draw[line width=1.5] plot[smooth, tension=1] coordinates {(1,0) (0.65,-0.2) (0.4,0.1) (0,-0.3)};
            \draw[line width=1.5] (0,0.7) -- (0,-2);
            \node[above left, blue] at (0, 0.7) {\large $z_1$};
            \node[below left, blue] at (0,-2) {\large $z_3$};
            \node[left, blue] at (-1, 0) {\large $z_2$};
            \node[right, blue] at (1, 0) {\large $z_4$};
            \node[left, blue] at (-0.5,0.5) {\large $L_2$};
            \node[right, blue] at (0.5,0.5) {\large $L_1$};
            \node[left, blue] at (-0.5,-1) {\large $L_3$};
            \node[right, blue] at (0.5,-1) {\large $L_4$};
            \node at (0.3,0.25) {\Huge $g_1$};
            \node at (-0.3,0.25) {\Huge $g_2$};
            \node at (0.3,-0.7) {\Huge $g_4$};
            \node at (-0.3,-0.7) {\Huge $g_3$};
        \end{scope}
        \end{tikzpicture}
        \caption{Sketch of the proof for Theorem~\ref{theo_kite}. \\Left: skeleton constructed from $S_{i,j}$'s on the left. \\Right: subharmonic extension of $g$ that generates the skeleton.}
        \label{fig:kite}
    \end{figure}

    We set the segment of $\ell$ inside the kite shape by $L$. Then the three curves $K_1,K_1'$ and $L$ divide $\Omega$ into four components at $c$. We shall denote the union of these four arcs by $K$. Then one can extend $g$ continuously to $\C \backslash K$ using $g_i$'s on the four components. That is, the extension coincides $g_i$ on the component touching $L_i$, which we denote by $\Omega_i$, which we denote by $\tilde g$. Note that on $\Omega_i \cup \Omega_{i+1}$, $\tilde g$ coincides $\max(g_i, g_{i+1})$ as $\Omega_i$ and $\Omega_{i+1}$ are contained in different nodal regions of $g_i - g_{i+1}$, which is $W_{i,i+1} \backslash S_{i,i+1}$. It follows that $\tilde g$ satisfies the mean value inequality on $K$ for all sufficiently small radii. This shows that $g$ after extension is globally subharmonic and continuous which is harmonic on $\C \backslash K$. Thus we have $\Delta \tilde g$ is a positive Borel measure compactly supported inside $K$. In particular, $\supp \Delta \tilde g$ does not contain simple loops. The result hence follows from Lemma~\ref{subhar}
\end{proof}

As one can see, the main idea of the proof is to construct the extension of $g$ using the mirror reflections $g_i$'s. So long as the switching between $g_i$ and $g_j$ happens only on $S_{i,j}$ and respects the nodal regions of $g_i-g_j$, we will end up with a subharmonic extension. The main challenge here is that in order for the subharmonic extension to admit the electrostatic skeleton, the "switching arcs" need to form a tree. This is easy in the case of kite shapes due to symmetry but much harder in general.

\subsection{Isosceles Trapezoids}
The simplicity of the proof to Theorem~\ref{theo_kite} stems from the fact that two arcs in the electrostatic skeleton are on $\ell$, while the other two arcs are symmetric. Therefore, we effectively only need to look at the behavior of one arc to construct the skeleton. This is a special property for kite shapes that unfortunately cannot be generalized. To prove the existence of the skeleton for isosceles trapezoids, we need the following proposition:

\begin{lemma}\label{trape}
    For any isosceles trapezoid $\Omega$, we parametrize each analytic curve $S_{i,i+1}$ by a function $\gamma_{i}$ with non-vanishing derivative oriented from $z_i$ to $\infty$. We denote the line of symmetric for $\Omega$ by $\ell$, if $t$ is a critical point of $g_i \circ \gamma_i$, then $z_i$ and $\gamma_i(t)$ are separated by $\ell$.
    
    In particular, before hitting $\ell$, $g_i$ is monotone along $\gamma_i$. And it makes sense to define the first time when $\gamma_i$ and $\gamma_{i+1}$ meet.
\end{lemma}

    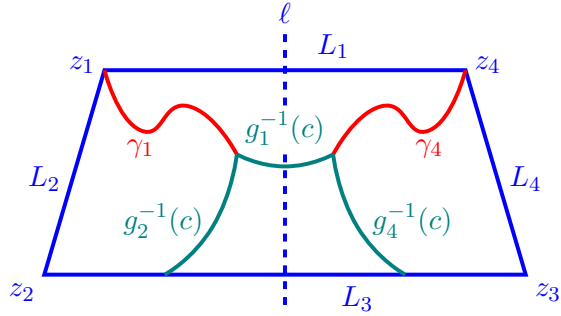
\begin{figure}[h!]
        \centering
        \begin{tikzpicture}[scale=1.6]
            \draw[blue, line width=1.5, dashed] (0,2) node[above]{\large $\ell$} -- (0,1.4);
            \draw[blue, line width=1.5, dashed] (0,1) -- (0,-0.3);
            \draw[blue, line width=1.5] (2,0) -- (1.5,1.7) -- (-1.5,1.7) -- (-2,0) -- cycle; 
            \draw[red, line width=1.5] plot[smooth, tension=1] coordinates {(1.5,1.7) (1.2, 1.2) (0.8,1.4) (0.4, 1)};
            \draw[red, line width=1.5] plot[smooth, tension=1] coordinates {(-1.5,1.7) (-1.2, 1.2) (-0.8,1.4) (-0.4, 1)};
            \draw[teal, line width=1.5] plot[smooth, tension=1] coordinates {(-1,0) (-0.6,0.4) (-0.4,1)};
            \draw[teal, line width=1.5] plot[smooth, tension=1] coordinates {(1,0) (0.6,0.4) (0.4,1)};
            \draw[teal, line width=1.5] plot[smooth, tension=1] coordinates {(-0.4,1) (0,0.9) (0.4,1)};
            \node[blue,above left] at (-1.5,1.6) {\large $z_1$};
            \node[blue,below left] at (-2,0) {\large $z_2$};
            \node[blue,above right] at (1.5,1.6) {\large $z_4$};
            \node[blue,below right] at (2,0) {\large $z_3$};
            \node[red,below] at (-1.2,1.2) {\large $\gamma_{1}$};
            \node[red,below] at (1.2,1.2) {\large $\gamma_{4}$};
            \node[teal, left] at (-0.6, 0.45) {\large $g^{-1}_2(c)$};
            \node[teal, right] at (0.65, 0.45) {\large $g^{-1}_4(c)$};
            \node[teal, above] at (0,1) {\large $g^{-1}_1(c)$};
            \node[blue] at (0.4,1.9) {\large $L_1$};
            \node[blue] at (0.6,-0.2) {\large $L_3$};
            \node[blue] at (2,0.8) {\large $L_4$};
            \node[blue] at (-2,0.8) {\large $L_2$};
        \end{tikzpicture}
        \caption{Sketch of proof of Lemma 3.1. If $t_0$ is a critical point of the real function $g_1\circ \gamma_1 = g_2\circ \gamma_1$ and $\gamma_1(t)$ is on the left of $\ell$, then the level sets $g_1^{-1}(c)$, $g_2^{-1}(c)$ and $g_4^{-1}(c)$ will intersect tangentially for $c = g_1(\gamma_1(t_0))$. This is geometrically impossible.}
        \label{fig:trapezoid}
    \end{figure}

\begin{proof}
    Without loss of generality, we assume that $z_1$ and $z_4$ are symmetric vertices for $\Omega$. Suppose for the contradiction that there exists a critical point $t_0$ for $g_1\circ \gamma_{1}$ such that $z_1$ and $\gamma_{1}(t_0)$ lie on the same side of $\ell$. We set $c = g_1(\gamma_{1}(t_0))$. Since $t_0$ is a critical point of $g_1\circ \gamma_1 = g_2\circ \gamma_1$, we have \begin{equation*}
        \nabla g_1(\gamma_1(t_0)) \cdot \gamma_1'(t_0) = \nabla g_2(\gamma_1(t_0)) \cdot \gamma_1'(t_0) = 0.
    \end{equation*} In particular, $\nabla g_1$ and $\nabla g_2$ are parallel at $\gamma_1(t_0)$. It follows that the convex curves $g_1^{-1}(c)$ and $g_2^{-1}(c)$ intersect tangentially at $\gamma_{1}(t_0)$. However, by symmetry, $g_1^{-1}(c)$ and $g_4^{-1}(c)$ also intersect tangentially at some point symmetric to $\gamma_{1}(t_0)$. This would allow us to construct the $C^1$ convex curve symmetric with respect to $\ell$ by combining the segments of three level sets, which starts and ends outside $\Omega$. This is impossible for that $\Omega$ is convex.

\end{proof}

Now we are ready to prove Theorem~\ref{theo_trap}. Essentially, Lemma~\ref{trape} guarantees that it makes sense to ask where $S_{1,2}$ and $S_{2,3}$ (parametrized by $\gamma_1$ and $\gamma_2$) meet for the first time before hitting $\ell$. This helps us separate cases and build the skeleton accordingly.

\begin{proof}[Proof of Theorem~\ref{theo_trap}]
    Given the isosceles trapezoid $\Omega$, we again denote the line of symmetry by $\ell$ and assume the setup in Lemma~\ref{trape}. We denote the component of $\Omega \backslash \ell$ that touches $L_1$ by $\Omega_1$. We also denote the component of $S_{1,2} \cap\Omega_1$ (paramtrized by $\gamma_{1}$) that contains $z_1$ by $S_1$ and that of $S_{2,3}\cap\Omega_1$ that contains $z_2$ by $S_2$. Then by Lemma~\ref{trape}, $g_1\circ \gamma_{1}$ (with respect to $g_2\circ \gamma_{2}$)  monontonely decreases when $\gamma_{1}(t) \in S_1$ (with respect to $\gamma_{2}(t) \in S_2$). 

    We will consider first the case when $S_1$ and $S_2$ are disjoint. We will denote their intersections with $\ell$ by $w_1$ and $w_2$ respectively. Then by symmetry, $\gamma_{1}$ and $\gamma_{4}$ meet at $w_1$ and $\gamma_{2}$ and $\gamma_{3}$ meet at $w_2$. Additionally, $w_1$ and $w_2$ are in turn connected by a segment of $S_{2,4}$, which we denote by $S'$. Let $S_1'$ and $S_2'$ be the reflection of $S_1$ and $S_2$ with respect to $\ell$ which contain part of $S_{1,4}$ and $S_{3,4}$ respectively. Then we can easily verify that \begin{equation*}
        K := S_1 \cup S_2 \cup S' \cup S_1' \cup S_2'
    \end{equation*} forms a tree which allows $g$ extends continuously to a subharmonic function globally.

    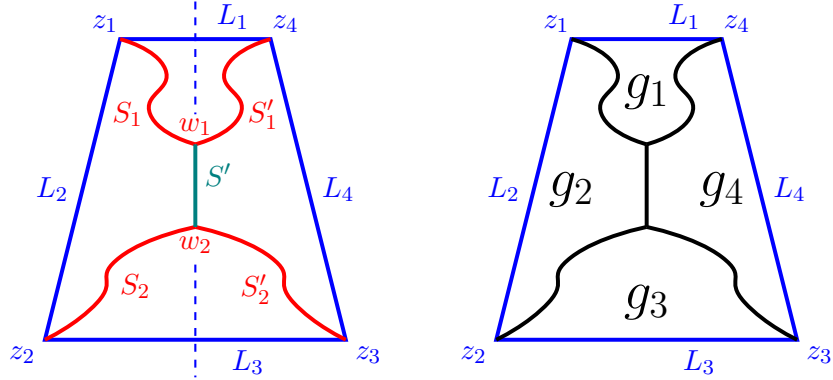
\begin{figure}[h!]
        \centering
        \begin{tikzpicture}
        \begin{scope}
            \draw[blue, line width=1.5] (2,0) -- (1,4) -- (-1,4) -- (-2,0) -- cycle;
            \draw[blue, line width=0.8, dashed] (0,4.5) -- (0,3);
            \draw[blue, line width=0.8, dashed] (0,1) -- (0,-0.5);
            \draw[red, line width=1.5] plot[smooth, tension=1] coordinates {(1,4) (0.4,3.6) (0.6,3) (0,2.6)};
            \draw[red, line width=1.5] plot[smooth, tension=1] coordinates {(-1,4) (-0.4,3.6) (-0.6,3) (0,2.6)};
            \draw[red, line width=1.5] plot[smooth, tension=1] coordinates {(2,0) (1.3,0.5) (1,1.1) (0,1.5)};
            \draw[red, line width=1.5] plot[smooth, tension=1] coordinates {(-2,0) (-1.3,0.5) (-1,1.1) (0,1.5)};
            \draw[teal, line width=1.5] (0,1.5) -- (0,2.6);
            \node[red, above] at (0,2.6) {\large $w_1$};
            \node[red, below] at (0,1.5) {\large $w_2$};
            \node[blue] at (1.2,4.2) {\large $z_4$};
            \node[blue] at (2.3, -0.2) {\large $z_3$};
            \node[blue] at (-1.2,4.2) {\large $z_1$};
            \node[blue] at (-2.3, -0.2) {\large $z_2$};
            \node[red, below] at (0.8,1) {\large $S_2'$};
            \node[red, below] at (-0.8,1) {\large $S_2$};
            \node[red] at (0.9,3) {\large $S_1'$};
            \node[red] at (-0.9,3) {\large $S_1$};
            \node[teal, right] at (0,2.2) {\large $S'$};

            \node[blue] at (0.5,4.3) {\large $L_1$};
            \node[blue] at (0.7,-0.3) {\large $L_3$};
            \node[blue] at (1.9,2) {\large $L_4$};
            \node[blue] at (-1.9,2) {\large $L_2$};
        \end{scope}

        \begin{scope}[shift={(6,0)}]
            \node[blue] at (0.5,4.3) {\large $L_1$};
            \node[blue] at (0.7,-0.3) {\large $L_3$};
            \node[blue] at (1.9,2) {\large $L_4$};
            \node[blue] at (-1.9,2) {\large $L_2$};
            
            \draw[blue, line width=1.5] (2,0) -- (1,4) -- (-1,4) -- (-2,0) -- cycle;
            \draw[line width=1.5] plot[smooth, tension=1] coordinates {(1,4) (0.4,3.6) (0.6,3) (0,2.6)};
            \draw[line width=1.5] plot[smooth, tension=1] coordinates {(-1,4) (-0.4,3.6) (-0.6,3) (0,2.6)};
            \draw[line width=1.5] plot[smooth, tension=1] coordinates {(2,0) (1.3,0.5) (1,1.1) (0,1.5)};
            \draw[line width=1.5] plot[smooth, tension=1] coordinates {(-2,0) (-1.3,0.5) (-1,1.1) (0,1.5)};
            \draw[line width=1.5] (0,1.5) -- (0,2.6);
            \node[blue] at (1.2,4.2) {\large $z_4$};
            \node[blue] at (2.3, -0.2) {\large $z_3$};
            \node[blue] at (-1.2,4.2) {\large $z_1$};
            \node[blue] at (-2.3, -0.2) {\large $z_2$};

            \node at (0,3.3) {\Huge $g_1$};
            \node at (0,0.5) {\Huge $g_3$};
            \node at (1,2) {\Huge $g_4$};
            \node at (-1,2) {\Huge $g_2$};
        \end{scope}
        \end{tikzpicture}
        \caption{First case. Left: Skeleton. Right: Subharmonic extension.}
        \label{fig:trape_1}
    \end{figure}

    Now we consider the case when $S_1$ intersects $S_2$. Note that if $S_1$ and $S_2$ only intersect on $\ell$, it can fall under the previous case such that $w_1$ and $w_2$ are the same point. Thus we assume otherwise. Since $\gamma_{1}$ and $\gamma_{2}$ are monotone on $S_1$ and $S_2$ respectively, we can order the points in $S_1\cap S_2$ in terms of the value of $g_1\circ \gamma_{1}$. We denote the point with the largest value as $w$. That is, $w$ is where two curves "meet for the first time". By definition, we have \begin{equation*}
        g_1(w) = g_2(w) = g_3(w)
    \end{equation*} and we denote this value by $a$. This implies that $w\in S_{1,3}$. Note that since $S_1$ and $S_2$ are distinct analytic curves, $S_1\cap S_2$ is a discrete set near $w$. Thus such a $w$ is unique. Let $K_3$ be the unique component of $S_{1,3}$ that connects $w$ to $\ell$. Note that this component is unique because due to symmetry, $S_{1,3}$ cannot intersect $\ell$ twice. 
    
    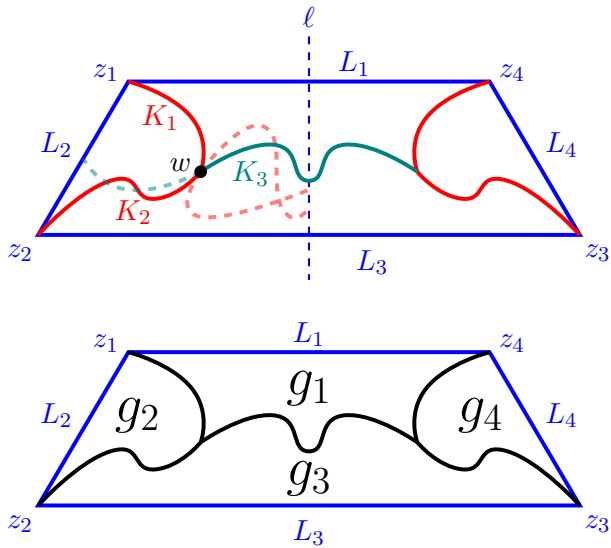
\begin{figure}[h!]
        \centering
        \begin{tikzpicture}[scale=1.2]
        \begin{scope}
            \draw[blue, line width=1.5] (3,0) -- (2,1.7) -- (-2,1.7) -- (-3,0) -- cycle;
            \draw[blue, line width=0.8, dashed] (0,2.2) node[above] {\large $\ell$} -- (0,-0.5);
            \draw[red, line width=1.5] plot[smooth, tension=1] coordinates {(3,0) (2.2,0.6) (1.7,0.4) (1.2,0.7)};
            \draw[red, line width=1.5] plot[smooth, tension=1] coordinates {(-3,0) (-2.2,0.6) (-1.7,0.4) (-1.2,0.7)};
            \draw[red, line width=1.5] plot[smooth, tension=1] coordinates {(2,1.7) (1.3,1.3) (1.2,0.7)};
            \draw[red, line width=1.5] plot[smooth, tension=1] coordinates {(-2,1.7) (-1.3,1.3) (-1.2,0.7)};
            \draw[red, line width=1.5, dashed, opacity=0.5] plot[smooth,tension=1] coordinates {(-1.2,0.7) (-0.5, 1.2) (-0.3, 0.3) (0, 0.3)};
            \draw[red, line width=1.5, dashed, opacity=0.5] plot[smooth,tension=1] coordinates {(-1.2,0.7) (-1.2, 0.2) (0, 0.5)};
            \draw[teal, line width=1.5, dashed, opacity=0.5] plot[smooth,tension=1] coordinates {(-1.2,0.7) (-2, 0.5) (-2.5, 0.85)};
            \draw[teal, line width=1.5] plot[smooth, tension=1] coordinates {(-1.2,0.7) (-0.4,1) (0, 0.6) (0.4,1) (1.2, 0.7)};
            \node[blue, above right=0.1pt] at (2, 1.6) {\large $z_4$};
            \node[blue] at (3.2, -0.2) {\large $z_3$};
            \node[blue, above left=0.1pt] at (-2, 1.6) {\large $z_1$};
            \node[blue] at (-3.2, -0.2) {\large $z_2$};
            \fill (-1.2, 0.7) circle (2pt);
            \node[above left] at (-1.2, 0.6) {\large $w$};
            \node[red] at (-1.65, 1.3) {\large $K_1$};
            \node[red] at (-1.95, 0.25) {\large $K_2$};
            \node[teal] at (-0.65, 0.68) {\large $K_3$};

            \node[blue] at (0.5,1.9) {\large $L_1$};
            \node[blue] at (0.7,-0.3) {\large $L_3$};
            \node[blue] at (2.8,1) {\large $L_4$};
            \node[blue] at (-2.8,1) {\large $L_2$};
        \end{scope}
        \begin{scope}[shift={(0,-3)}]
            \draw[blue, line width=1.5] (3,0) -- (2,1.7) -- (-2,1.7) -- (-3,0) -- cycle;
            \draw[line width=1.5] plot[smooth, tension=1] coordinates {(3,0) (2.2,0.6) (1.7,0.4) (1.2,0.7)};
            \draw[line width=1.5] plot[smooth, tension=1] coordinates {(-3,0) (-2.2,0.6) (-1.7,0.4) (-1.2,0.7)};
            \draw[line width=1.5] plot[smooth, tension=1] coordinates {(2,1.7) (1.3,1.3) (1.2,0.7)};
            \draw[line width=1.5] plot[smooth, tension=1] coordinates {(-2,1.7) (-1.3,1.3) (-1.2,0.7)};
            \draw[line width=1.5] plot[smooth, tension=1] coordinates {(-1.2,0.7) (-0.4,1) (0, 0.6) (0.4,1) (1.2, 0.7)};
            \node[blue, above right=0.1pt] at (2, 1.6) {\large $z_4$};
            \node[blue] at (3.2, -0.2) {\large $z_3$};
            \node[blue, above left=0.1pt] at (-2, 1.6) {\large $z_1$};
            \node[blue] at (-3.2, -0.2) {\large $z_2$};

            \node[blue] at (0,1.9) {\large $L_1$};
            \node[blue] at (0,-0.3) {\large $L_3$};
            \node[blue] at (2.8,1) {\large $L_4$};
            \node[blue] at (-2.8,1) {\large $L_2$};

            \node at (1.9,1) {\Huge $g_4$};
            \node at (-1.9,1) {\Huge $g_2$};
            \node at (0,1.3) {\Huge $g_1$};
            \node at (0,0.3) {\Huge $g_3$};
        \end{scope}
        \end{tikzpicture}
        
        \caption{Second case. Top: Skeleton. Bottom: Subharmonic extension}
        \label{fig:trape2}
    \end{figure}

    We set $K_1$ and $K_2$ to be the images of $\gamma_{1}$ and $\gamma_{2}$ up to $w$ respectively. Then for any $z\in K_1\cap K_3$, we also have $g_1(z) = g_2(z) =g_3(z)$, but the only point on $K_1$ that would satisfy this equality is $w$. Thus we have \begin{equation*}
        K_i \cap K_j = \set{w} , \quad i\neq j \in \set{1,2,3}.
    \end{equation*}
    We now set $K_1', K_2'$ and $K_3'$ to be the reflections of $K_1, K_2$ and $K_3$ against $\ell$ respectively. Then the compact set \begin{equation*}
        K:= K_1\cup K_2 \cup K_3 \cup K_1'\cup K_2' \cup K_3'
    \end{equation*} forms a tree which allows $g$ to extend to a globally continuous subharmonic function. The result then follows.
\end{proof}

\section{Strict descent condition}\label{sec_sd}
\subsection{Motivating example: generic quadrilateral}
It is still an open problem if every convex quadrilateral admits an electrostatic skeleton. Nevertheless, if we set $g_{\max} :=\max\set{g_1,g_2,g_3,g_4}$, numerical simulations suggest that the measure $(2\pi)^{-1}\Delta g_{\max} $ has a tree-like support, making such a measure the skeleton. In this section, we will explore the structure of this measure, which will motivate the algorithm used to construct the electrostatic skeleton under the condition of strict descent.

The important observation here is that the support of $\Delta g_{\max}$ consists of corners of the level sets $g_{\max}^{-1}(-t)$, where $t>0$ is small enough that the level set is an analytic quadrilateral, with each side being a segment of $g_i^{-1}(t)$. As we increase $t$, one of the following three things will happen to $g_{\max}^{-1}(-t)$:
\begin{figure}[h!]
    \centering
    \includegraphics[width=0.5\linewidth]{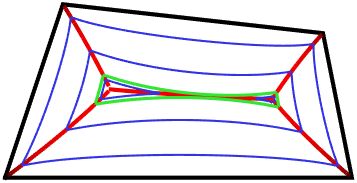}
    \caption{Shrinking level sets in a thin quadrilateral, phase transition highlighted in green}
    \label{fig:placeholder1}
\end{figure}

\begin{enumerate}
    \item it vanishes to a single point (convex kite shapes),
    \item it splits into two analytic triangles (quadrilaterals that are thin, see Figure~\ref{fig:placeholder1}),
    \item it converts into one single analytic triangle (quadrilaterals with a significantly shorter side, see Figure~\ref{fig:placeholder2}).
\end{enumerate}

\begin{figure}[h!]
    \centering
    \includegraphics[width=0.5\linewidth]{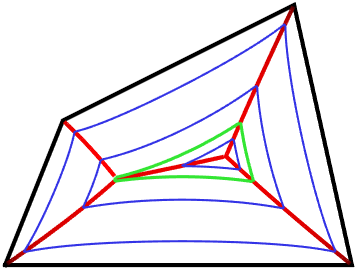}
    \caption{Shrinking level sets in a quadrilateral with a significantly shorter side, phase transition highlighted in green}
    \label{fig:placeholder2}
\end{figure}

It is worth paying attention to the critical time when the phase transitions of the latter two cases happen. For the second case, the phase transition happens when two non-adjacent analytic arcs meet. For the third case, it happens when two adjacent vertices meet. In either case, the vertices of the analytic triangles are still in the support in $\Delta g_{\max}$

We note that the analytic quadrilaterals and triangles that arise as the level sets of $g_{\max}$ share some geometric properties:

\begin{enumerate}
    \item they are both piecewise analytic Jordan domains enclosed by finitely many segments of level sets $g_i^{-1}(-t)$'s
    for a fixed $t\geq 0$,
    \item on the segment of level set $g_i^{-1}(-t)$, the gradient $\nabla g_i$ points toward the outside of the domain (the edges concave inwards),
    \item at each vertex, the internal angles of the domain is strictly less than $\pi$.
\end{enumerate}

\begin{figure}[h!]
    \centering
    \includegraphics[width=0.5\linewidth]{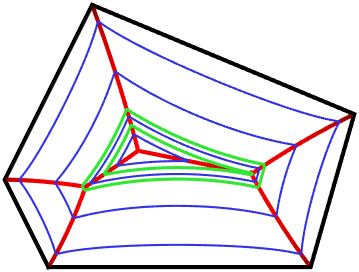}
    \caption{Multiple phase transitions occur when shrinking the level sets inside the pentagon}
    \label{fig:placeholder3}
\end{figure}

We will call the analytic polygons that satisfy these three properties the \textit{regular loops} of the polygon. The vertices of the regular loops will be the building blocks for electrostatic skeletons. The main idea of the skeleton building algorithm is to find a subharmonic extension $\tilde{g}$ of $g$ inside the convex polygon that satisfies the following:

\begin{quote}
    the level set $\tilde{g}^{-1}(-t)$ is a regular loop which shrinks as we increase $t$. The level set $\tilde{g}^{-1}(-t)$ will go through phase transitions and decompose into smaller regular loops (see Figure~\ref{fig:placeholder3}). After finitely many phase transitions, all the regular loops vanish to single points.
\end{quote}
We define a \textit{critical loop} as the curve that arises during the phase transition of a regular loop. The more rigorous definitions of regular and critical loops will be introduced in Section~\ref{algorithm} after we introduce the proper notations.

The condition of strict descent is posed to assure that during the shrinking, the internal angles of the loops remain strictly less than $\pi$. This, in practice, always happens for the level sets of $\max\set{g_1,\dots,g_n}$. However, showing such a property for $\max\set{g_1,\dots,g_n}$ (or any subharmonic extensions of $g$) is difficult and might require delicate estimates of Schwarz–Christoffel maps. Thus in this paper, we will focus on showing the condition of strict descent implies the existence of electrostatic skeleton, instead of proving the condition.

\subsection{Parametrization of the zero sets}
As we can see in the proof of Theorem~\ref{theo_trap}, it is convenient to construct the electrostatic skeleton when we can control when two zero sets intersect for the first time. In this section, we will look further into the geometry of the zero sets with the strict descent condition. Again we assume the notations from Section~\ref{notation}.

\begin{lemma}\label{unique_max}
    The Schwarz reflection $g_i$ has a unique maximum on $S_{i,j}$ for all $j\neq i$. 
\end{lemma}
\begin{proof}
    If $L_i$ and $L_j$ are adjacent edges (w.l.o.g. we assume $j=i+1$), it is clear that $g_i<0$ in the interior of $W_{i,j}$ and $g_i(z_i) = 0$. So the result follows trivially. Now suppose $L_i$ and $L_j$ are not adjacent. We note that the function \begin{equation*}
        t \mapsto d(\set{z:g_i(z)\geq-t}, \set{z:g_j(z) \geq -t})
    \end{equation*} is continuous. Let $t_0>0$ be the first time this decreasing function is zero. Then $\set{z:g_i(z)\geq-t}$ and $\set{z:g_j(z) \geq -t}$ intersect when $t=t_0$ and are disjoint when $t<t_0$. It follows that $g_i^{-1}(-t_0)$ and $g_j^{-1}(-t_0)$ are tangent. By Lemma~\ref{convex}, $g_i^{-1}(-t_0)$ and $g_j^{-1}(-t_0)$ are convex curves and thus can only be tangent at a unique point. That tangent point is thus the global maximum of both $g_i$ and $g_j$ on $S_{i,j}$.
\end{proof}

The reflection functions $g_i$ and $g_j$ share the same global maximum on $S_{i,j}$. We shall denote this maximum by $m_{i,j}$. When $j=i+1$, the unique maximum is $z_i=m_{i,i+1}$. When $L_i$ and $L_j$ are disjoint, $m_{i,j}$ is in the interior of $W_{i,j}$. Lemma~\ref{no_branch} and Lemma~\ref{unique_max} together imply the following. 

\begin{proposition}[Strict descent condition reformatted]\label{sd_con}
    Given a convex polygon $\Omega$ that satisfies the \textit{strict descent condition}, the following conditions hold: \begin{enumerate}
        \item for any pair of adjacent edges $L_i$ and $L_{i+1}$, we can parametrize $S_{i,i+1}$ as a curve $\gamma_{i,i+1}$ such that \begin{equation*}
            \gamma_{i,i+1}(0) = z_i, \quad g_i(\gamma_{i,i+1}(t)) = g_{i+1}(\gamma_{i,i+1}(t)) = -t,
        \end{equation*}
        \item for any pair of non-adjacent edges $L_i$ and $L_j$, let $m_{i,j}$ be the unique global maximum shared by $g_i$ and $g_j$ on $S_{i,j}$ and set $t_{i,j} = -g_i(m_{i,j})$. Then we can parametrize $S_{i,j}$ as two curves $\gamma_{i,j}$ and $\eta_{i,j}$ meeting at $m_{i,j}$ such that\begin{align*}
            g_i(\gamma_{i,j}(t)) = g_j(\gamma_{i,j}(t)) =g_i(\eta_{i,j}(t)) = g_j(\eta_{i,j}(t)) = -t,\\
            \gamma_{i,j}(t_{i,j}) = \gamma_{i,j}(t_{i,j}) = \eta_{i,j}(t_{i,j}) = \eta_{i,j}(t_{i,j}) = m_{i,j}.
        \end{align*}
    \end{enumerate}
\end{proposition}
\begin{proof}
    We first note that for $z\in S_{i,j}$, if $\nabla g_i(z)$ and $\nabla g_j(z)$ are pointed at the exact opposite direction, then have that $g_i^{-1}(g_i(z))$ and $g^{-1}_j(g_i(z))$ must be tangential at $z$. This implies that $z= m_{i,j}$. Therefore the condition of strict descent as in Definition~\ref{con_sd1} implies that on $S_{i,j}\backslash m_{i,j}$, $\nabla g_i$ and $\nabla g_j$ cannot be parallel.
    
    Note that away from the global maximum $m_{i,j}$ of $g_i$ (and $g_j$) on $S_{i,j}$, we can parametrize $S_{i,j}$ by the smooth curve $\beta$. Then we have \begin{equation*}
        \pder{}{t}g_i(\beta(t)) = \nabla g_i(\beta(t))\cdot \beta'(t) = \nabla g_j(\beta(t))\cdot \beta'(t).
    \end{equation*} 
    Since $S_{i,j}$ is smooth, we can assume that $\beta'(t)$ is nowhere vanishing. Since $\beta(t)$ is away from $m_{i,j}$, $\nabla g_i(z)$ and $\nabla g_j(z)$ are not parallel. Thus we have \begin{equation*}
        \nabla g_i(\beta(t))\cdot \beta'(t) = \nabla g_j(\beta(t))\cdot \beta'(t) \neq 0.
    \end{equation*} This implies that $g_i$ and $g_j$ are strictly monotone on $S_{i,j}$ away from $m_{i,j}$. Thus the result follows.
\end{proof}

For the notational convenience, we always assume $\gamma_{i,j}$ parametrizes one of the unbounded component of $S_{i,j}\backslash \set{m_{i,j}}$. In the case where both components of $S_{i,j}\backslash \set{m_{i,j}}$ are unbounded (as is the case where $L_i$ and $L_j$ are parallel), we assign $\gamma_{i,j}$ and $\eta_{i,j}$ arbitrarily. In consequence, the domain of $\gamma_{i,j}$ is always $[t_{i,j},\infty)$ while if $\eta_{i,j}$ exists, its domain is a half-open subinterval of $[t_{i,j},\infty)$ containing $t_{i,j}$. 

The strict descent condition thus significantly restricts the behavior of $S_{i,j}$'s and how they intersect one another. Under such a condition, given two adjacent edges $L_i$ and $L_{i+1}$ and any $t>0$, $\gamma_{i,i+1}(t)$ is the unique intersection between $g_i^{-1}(-t)$ and $g^{-1}_j(-t)$. Meanwhile, the condition of strict descent also implies that for any two non-adjacent edges $L_i$ and $L_j$, the images of $\eta_{i,j}$ and $\gamma_{i,j}$ are separated by \begin{equation*}
    g_i^{-1}([-t_{i,j},0]) \cup g_i^{-1}([-t_{i,j},0])
\end{equation*} since the two sets above intersect uniquely at $m_{i,j}$ (see Figure~\ref{fig:inter}). 

\begin{figure}[h!]
    \centering
    \begin{tikzpicture}
        \draw[blue, line width=1.5] plot[smooth, tension=1] coordinates {(-2.4, 0.6) (-2,0.6) (-0.9,1) (-0.4,0.7) (0.1,0.9)};

        \draw[blue, line width=1.5] plot[smooth, tension=0.4] coordinates {(-2.4,0.55) (-2.9,0.5) (-3.4,0.5) (-4.8,0.4) (-6.5,0)};
        \node[blue,left] at (-5,0.6) {\large $\eta_{i,j}$};
        
        \fill[teal, opacity=0.3] plot[smooth, tension=1] coordinates {(-3.9,0) (-2.2,0.57) (-0.9,0)};
        \fill[teal, opacity=0.3] plot[smooth, tension=1] coordinates {(-4.1,{-4.1*(1/3)+2.2}) (-2.2,0.63) (-1,{-1*(1/3)+2.2})};
        \fill[teal] (-2.4,0.55) circle (2.5pt);
        \node[black,above] at (-2.35,1.6) {$g_j^{-1}([-t_{i,j},0])$};
        \node[black,above] at (-2.35,-0.8) {$g_i^{-1}([-t_{i,j},0])$};
        \node[black,above] at (-2.4,0.6) {\large $m_{i,j}$};
        \node[blue,right] at (0.1,0.9) {\large $\gamma_{i,j}$};
  
        \draw[red, line width=1.5] (-7,0) -- (-0.5,0) node[red, right] {\large $\ell_i$};
        \draw[red, line width=1.5] (-7.1,-0.2) -- (-0.5,2) node[red, right] {\large $\ell_j$};
        
    \end{tikzpicture}
    \caption{Sketch of $J^\gamma_{i,j}(t)$ and $U^\gamma_{i,j}(t)$}
    \label{fig:inter}
\end{figure}

\begin{corollary}[Geometry of the level sets] \label{geo_level}
    Under the condition of strict descent, consider the level set $g^{-1}_i(-t)$ inside the wedge $W_{i,j}$, \begin{enumerate}
        \item When $0\leq t< t_{i,j}$, $g^{-1}_i(-t)$ is completely contained in one of the components of $W_{i,j}\backslash S_{i,j}$.
        \item When $t=t_{i,j}$, $g^{-1}_i(-t)$ is completely contained in the closure of one of the components of $W_{i,j}\backslash S_{i,j}$, and intersects $S_{i,j}$ only at $m_{i,j}$.
        \item When $t>t_{i,j}$ and $\eta_{i,j}(t)$ is well-defined, $g^{-1}_i(-t)$ consists of one segment connecting $\ell_i$ to $\gamma_{i,j}(t)$, one segment connecting $\ell_i$ to $\eta_{i,j}(t)$ and one connecting $\gamma_{i,j}(t)$ and $\eta_{i,j}(t)$. The first two segments are separated from the last by $S_{i,j}$.
        \item When $t>t_{i,j}$ and $t$ is outside the domain of $\eta_{i,j}$, $g_i^{-1}(-t)$ consists of one segment connecting $\ell_i$ to $\gamma_{i,j}(t)$ and the other connecting $\gamma_{i,j}(t)$ to $\ell_j$. The two are separated by $S_{i,j}$.
    \end{enumerate}
\end{corollary}

In particular, if we assume the strict descent condition, then for any $t> t_{i,j}$, $g^{-1}_i(-t)$ will only intersect the image of $\gamma_{i,j}$ exactly once at $\gamma_{i,j}(t)$  (and that of $\eta_{i,j}$ at $\eta_{i,j}(t)$ if it exists). Let $\beta\in \set{\gamma,\eta}$. If $t$ is in the domain of $\beta_{i,j}$, it follows that one of the components of \begin{equation*}
    W_{i,j} \backslash \left(g_i^{-1}([-t,0])\cup g_j^{-1}([-t,0])\right)
\end{equation*} contains the entirety of the curve $\beta_{i,j}$ after time $t$. We denote that component as $U^\beta_{i,j}(t)$ and set \begin{equation*}
    J^\beta_{i,j}(t) = \partial U^\beta_{i,j}(t) \backslash (\ell_i \cup \ell_j).
\end{equation*} 
Then $J^\beta_{i,j}(t)$ is a curve consisting of a segment of $g_i^{-1}(-t)$ connecting $\ell_i$ to $\beta_{i,j}(t)$ and that of $g_j^{-1}(-t)$ connecting $\ell_j$ to $\beta_{i,j}(t)$. It also cuts $W_{i,j}$ into two components (see Figure \ref{fig:inward}).
    
\begin{figure}[h!]
    \centering
    \begin{tikzpicture}
        
        \fill[orange, opacity=0.2] (-0.5,2) -- (-1.3,{-1.3*(1/3)+2.2}) -- (-1.3,1.2) -- (-1.7,0.7)-- (-1.3,0.55) -- (-1.25,0.4) -- (-1.3, 0) -- (-0.5,0) arc(-30:30:2)-- cycle;

        \fill[orange, opacity=0.2] (4,2) -- (3.5,1.5) -- (3.6, 1.2) -- (3.5,0.9) --(3.8,0.5) --  (3.9,0) -- (5,0) arc(0:45:3) -- cycle;
        
        \draw[red, line width=1.5] plot[smooth, tension=1] coordinates {(-2.2,0.6) (-0.9,0.8) (-0.4,0.7) (0.1,0.9)};

        \draw[red, line width=1.5] plot[smooth, tension=1] coordinates {(-2.25,0.6) (-2.9,0.5) (-3.4,0.5) (-3.8,0.4)};
        \node[red,left] at (-3.8,0.4) {\large $\eta_{i,j}$};
        
        \draw[red, line width=1.5] plot[smooth, tension=1] coordinates {(2,0) (2.7,0.2) (3.5,0.9) (4.3,0.8) (5,1.6)};
        \draw[teal, line width=1.5] plot[smooth, tension=1] coordinates {(-3.2,0) (-2.2,0.6) (-1.5,0)};
        \draw[teal, line width=1.5] plot[smooth, tension=1.5] coordinates {(-3.1,{-3.1*(1/3)+2.2}) (-2.2,0.6) (-1.5,{-1.5*(1/3)+2.2})};
        \fill[teal] (-2.25,0.6) circle (2.5pt); 
        \fill[teal] (2,0) circle (2.5pt);
        \node[teal,above] at (-2.35,0.6) {\large $m_{i,j}$};
        \node[teal,below] at (2,0) {\large $m_{i,i+1} = z_i$};
        
        \draw[blue, line width=1.5] plot[smooth, tension=1] coordinates {(-1.3,{-1.3*(1/3)+2.2}) (-1.3,1.2) (-1.7,0.7)};
        \draw[blue, line width=1.5] plot[smooth, tension=1] coordinates {(-1.3,0) (-1.3,0.5) (-1.7,0.7)};
        \fill[blue] (-1.7,0.7) circle (2.5pt);
        \draw[blue, line width=1.5] plot[smooth, tension=1] coordinates {(3.5,1.5) (3.6, 1.2) (3.5,0.9)};
        \draw[blue, line width=1.5] plot[smooth, tension=1] coordinates {(3.9,0) (3.8,0.5) (3.5,0.9)};
        \fill[blue] (3.5,0.9) circle (2.5pt); 
        \node[red,right] at (0.1,0.9) {\large $\gamma_{i,j}$};
        \node[red,right] at (5,1.6) {$\gamma_{i,i+1}$};
        \node[blue,below] at (-1.3,0) {\small $J^\gamma_{i,j}(t)$};
        \node[blue,below] at (3.9,0) {\small $J^\gamma_{i,i+1}(t)$};
        \node[orange] at (-0.6,1.4) {\Large $U^\gamma_{i,j}(t)$};
        \node[orange] at (4.7,0.4) {\Large $U^\gamma_{i,i+1}(t)$};

        \draw[red, line width=1.5] (-3.5,0) -- (-0.5,0) node[red, right] {\large $\ell_i$};
        \draw[red, line width=1.5] (-3.5,1) -- (-0.5,2) node[red, right] {\large $\ell_j$};
        \draw[red, line width=1.5] (2,0) -- (5,0) node[red, right] {\large $\ell_{i+1}$};
        \draw[red, line width=1.5] (2,0) -- (4.15,2.1) node[red, right] {\large $\ell_i$};
    \end{tikzpicture}
    \caption{Sketch of $J^\gamma_{i,j}(t)$ and $U^\gamma_{i,j}(t)$}
    \label{fig:inward}
\end{figure}

Note that $J^{\beta}_{i,j}$ is a continuous function on the space of compact sets on $\C$ under the Hausdorff topology. Moreover, $J^{\beta}_{i,j}(t_{i,j}):= \lim_{t\downarrow t_{i,j}}J^\beta_{i,j}(t)$ is either $L_i \cup L_j$ (when $L_i$ and $L_j$ are adjacent), or consisting of parts of $g^{-1}_i(-t_{i,j})$ and $g^{-1}_j(-t_{i,j})$ forming a cusp. The next lemma gives us some geometric intuition about the curve $J^\beta_{i,j}$ and domain $U^\beta_{i,j}$.

\begin{lemma}\label{cuts}
    Under the condition of strict descent as in Proposition~\ref{sd_con}, let $t\geq t_{i,j}$ be such that $\beta_{i,j}(t)$ is a well-defined point in $W_{i,j}$. Then the angle $J^\beta_{i,j}(t)$ makes toward $U^{\beta}_{i,j}(t)$ is strictly less than $\pi$.
\end{lemma}

\begin{proof}
    It is easy to check that in the case when $t = t_{i,j}$, the angle $J_{i,j}^\beta(t)$ makes is either $0$ (when $1<|i-j|<n-1$) or an inner angle of the polygon (when $|i-j|\in \set{1,n-1}$). Thus we assume $t> t_{i,j}$. Note that since $W_{i,j} \backslash S_{i,j}$ has two components, the one touching $\ell_i$, which we denote by $\Omega_i$, satisfies \begin{equation*}
        \Omega_i = \set{z\in W_{i,j}:g_i(z)>g_j(z)}.
    \end{equation*}
    Note that the curve $\beta_{i,j}$ is smooth. We can let $T_{i,j}(t)$ be the normalized velocity vector of the curve at $\beta_{i,j}(t)$ and $N_{i,j}(t)$ the normal vector at $\beta_{i,j}(t)$ pointing toward $\Omega_i$. Then the density of $ \Delta \max(g_i,g_j)$ at $\beta_{i,j}(t)$ is given by \begin{equation*}
        -\nabla g_i(\beta_{i,j}(t)) \cdot N_{i,j}(t) + (-\nabla g_j(\beta_{i,j}(t)) \cdot (- N_{i,j}(t)) = \nabla (g_i-g_j)(\beta_{i,j}(t)) \cdot N_{i,j}(t).
    \end{equation*}
    This quantity should be non-negative as $\max(g_i,g_j)$ is subharmonic at $\beta_{i,j}(t)$. Meanwhile, the strict descent condition implies $\nabla g_i$ and $\nabla g_j$ are not parallel at $\beta_{i,j}(t)$. Thus \begin{equation*}
        \nabla g_i(\beta_{i,j}(t)) \cdot T_{i,j}(t) = \nabla g_j(\beta_{i,j}(t)) \cdot T_{i,j}(t) < 0.
    \end{equation*} 
    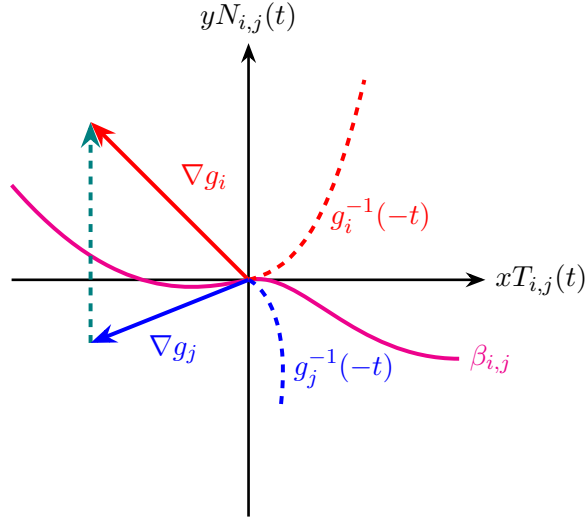
\begin{figure}[h!]
        \centering
    \begin{tikzpicture}[>=Stealth, scale=0.7,line width=1.5pt]
    \draw[->, black, line width=1pt] (-4.5,0) -- (4.5,0) node[right] {\large $x T_{i,j}(t)$};
    \draw[->, black, line width=1pt] (0,-4.5) -- (0,4.5) node[above] {\large $y N_{i,j}(t)$};

    \draw[magenta] (-4.5,1.8) 
        .. controls (-2.5,-0.5) and (-1,-0.2) .. (0,0)
        .. controls (1,0.2) and (2,-1.5) .. (4,-1.5) 
        node[right] {\large $\beta_{i,j}$};

    \draw[->, teal, dashed] (-3, -1.2) -- (-3, 3);
    \draw[->, red] (0,0) -- (-3, 3) node[pos=0.5, above right=0.1pt] {\large $\nabla g_i$};
    \draw[->, blue] (0,0) -- (-3, -1.2) node[pos=0.7, below right] {\large $\nabla g_j$};

    \draw[red, dashed] (0,0) 
        .. controls (0.8,0.2) and (1.5,0.8) .. (2.2, 3.8) 
        node[right, pos=0.6] {\large $g_i^{-1}(-t)$};

    \draw[blue, dashed] (0,0) 
        .. controls (0.4,-0.2) and (0.8,-1) .. (0.6, -2.5) 
        node[right, pos=0.8] {\large $g_j^{-1}(-t)$};
\end{tikzpicture}
        \caption{Coordinate plane by $T_{i,j}(t)$ and $N_{i,j}(t)$. Note that $\nabla g_i$ and $\nabla g_j$ do not necessarily lie in different quadrants}
        \label{fig:coord}
    \end{figure}
    
    It follows that in the coordinate plane $(T_{i,j}(t),N_{i,j}(t))$, $\nabla g_i(\beta_{i,j}(t))$ has the same $x$-coordinate as $\nabla g_j(\beta_{i,j}(t))$ but strictly larger $y$-coordinate (see Figure~\ref{fig:coord}). The statement hence follows.
\end{proof}

\section{Skeleton building algorithm} \label{algorithm}
In this section, we will prove Theorem~\ref{theo_descent} to show that the strict descent condition implies the existence of electrostatic skeleton. To prove the theorem, we first need to properly define the concepts of regular and critical loops.

\subsection{Regular loops}
Note that if $t>0$ is in the domains of both $\gamma_{i,j}$ and $\gamma_{j,k}$, then $\gamma_{i,j}(t)$ and $\gamma_{j,k}(t)$ are two points on $g_j^{-1}(-t)$. We shall denote the closed segment of $g_j^{-1}(-t)$ connecting the two points by the tuple $(t, \gamma_{i,j}, g_j, \gamma_{j,k})$ (see Figure~\ref{fig:arc}).

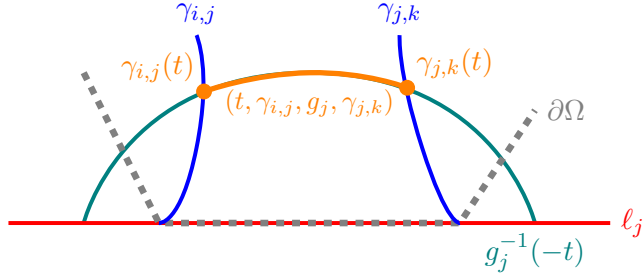
\begin{figure}[h!]
    \centering
    \begin{tikzpicture}
        \draw[teal, line width=1.5pt] plot[smooth, tension=1.5] coordinates {(-3,0) (0,2) (3,0)} node[teal, below] {\large $g_j^{-1}(-t)$};
        \draw[red, line width=1.5] (-4,0) -- (4,0) node[red, right] {\large $\ell_j$};
        \draw[gray, line width=2.5pt, dashed] (-3,2) -- (-2,0) -- (2,0) -- (3,1.5) node[gray, right] {\large $\partial\Omega$};
        \draw[blue, line width=1.5pt] plot[smooth, tension=1.5] coordinates {(-2,0) (-1.5,1) (-1.5,2.5)} node[blue, above] {\large $\gamma_{i,j}$};
        \draw[blue, line width=1.5pt] plot[smooth, tension=1.5] coordinates {(2,0) (1.5,1) (1.2,2.5)} node[blue, above] {\large $\gamma_{j,k}$};
        
        \draw[orange, line width=2pt] plot[smooth, tension=1] coordinates {(-1.4,1.75) (0, 2) (1.3,1.8)};
        \node[orange] at (0, 1.6) {$(t,\gamma_{i,j}, g_j, \gamma_{j,k})$};
        \fill[orange] (-1.4,1.75) circle (3pt) node[orange, above left] {\large $\gamma_{i,j}(t)$};
        \fill[orange] (1.3,1.8) circle (3pt) node[orange, above right] {\large $\gamma_{j,k}(t)$};

    \end{tikzpicture}
    \caption{Sketch of $(t,\beta_{i,j}, g_j, \beta_{j,k})$, in the case when $i+1=j=k-1$}
    \label{fig:arc}
\end{figure}

Similarly, we can define $(t, \gamma_{i,j}, g_j, \eta_{j,k})$ and $(t, \eta_{i,j}, g_j, \gamma_{i,k})$. We set \begin{align}
    (t,\beta_{i_0,i_1}, g_{i_1}, \beta_{i_1,i_2}, g_{i_2},\dots,\beta_{i_{n-1},i_n},g_{i_n},\beta_{i_n,i_{n+1}}) \nonumber\\= \bigcup_{k=1}^n (t,\beta_{i_{k-1},i_k},g_{i_k},\beta_{i_k,i_{k+1}}), \label{loops}
\end{align} where $\beta \in \set{\gamma,\eta}$ and we assume $t$ is in the domain of each $\beta_{i_k,i_{k+1}}$. When $\beta_{i_0,i_1}= \beta_{i_n,i_{n+1}}$, this is a loop that consists of segments of level sets of $g_i$'s. To simplify the notation, we also write the left-hand side of $\eqref{loops}$ as $(t,\bm{\beta},\bm{g})$ where \begin{align*}
    \bm{\beta} &= (\beta_{i_0,i_1},  \beta_{i_1,i_2},\dots,\beta_{i_n,i_{n+1}}),\\
    \bm{g} &= (g_{i_1}, g_{i_2},\dots, g_{i_n}).
\end{align*}

\begin{definition}[Regular loops]\label{reg_loop}
    Given a tuple $(t,\bm{\beta},\bm{g})$ such that $\beta_{i,j}(t)$ is well-defined for each $\beta_{i,j}\in\bm{\beta}$ and the first and last elements of $\bm{\beta}$ coincide (hence it represents a loop), we say $(t,\bm{\beta},\bm{g})$ represents a \textit{regular loop} if all of the following conditions are satisfied:
    \begin{enumerate}
        \item The loop represented by the tuple is a Jordan curve and $\bm{g}$ orders the arcs counter-clockwise, and \label{counter}
        \item Each arc $(t,\beta_{i,j},g_j,\beta_{j,k})$ has positive length ($\beta_{i,j}(t)\neq \beta_{j,k}(t)$), and \label{pos}
        \item On each arc $(t,\beta_{i,j},g_j,\beta_{j,k})$, $\nabla g_j$ points outside the loop, and \label{concave}
        \item For each $\beta_{i,j}\in \bm{\beta}$, the internal angle of the loop at $\beta_{i,j}(t)$ is in $[0,\pi)$. \label{acute}
    \end{enumerate}
\end{definition}

It is easy to see that the arc $(t,\beta_{i,j},g_j,\beta_{j,k})$ is a continuous function over $t$ in the space of compact sets on $\C$ under the Hausdorff topology. As a consequence, if $(t,\bm{\beta},\bm{g})$ is well-defined for $[a,b]$ and $(a,\bm{\beta}, \bm{g})$ represents a loop, $t\mapsto (t,\bm{\beta},\bm{g})$ is a continuous function on $[a,b]$ of loops as well. Lemma~\ref{cuts} implies an equivalent definition of a regular loop:

\begin{corollary}\label{equiv}
    Given a tuple $(t,\bm{\beta},\bm{g})$ that represents a loop and satisfies Conditions~\eqref{counter}, \eqref{pos} and \eqref{concave} in Definition~\ref{reg_loop}, $(t,\bm{\beta},\bm{g})$ satisfies Condition~$\eqref{acute}$ if and only if for each pair of adjacent arcs $(t,\beta_{i,j},g_j,\beta_{j,k})$ and $(t,\beta_{j,k},g_k,\beta_{k,l})$ intersecting at $\beta_{j,k}(t)$, we have \begin{equation} \label{contain}
        (t,\beta_{i,j},g_j,\beta_{j,k})\cup(t,\beta_{j,k},g_k,\beta_{k,l}) \subseteq J^\beta_{j,k}(t).
    \end{equation}
\end{corollary}

\begin{proof}
    Without loss of generality, we assume that the loop goes from $\beta_{i,j}(t)$ to $\beta_{j,k}(t)$ then to $\beta_{k,l}(t)$ in a counterclockwise fashion. Thus by Conditions~\eqref{counter} and \eqref{concave}, both $\nabla g_j$ on $(t,\beta_{i,j},g_j,\beta_{j,k})$ and $\nabla g_k$ on $(t,\beta_{j,k},g_k,\beta_{k,l})$ point to the right of the curve. Note that Corollary~\ref{geo_level} implies that under the strict descent condition, for any $t > t_{i,j} = -\sup_{z\in S_{i,j}} g_i(z)$, the level set $g_i^{-1}(-t)$ and $g_j^{-1}(-t)$ intersect at most twice, both in a non-tangential manner. It thus follows that $\eqref{contain}$ holds if and only if the internal angle at $\beta_{j,k}(t)$ is strictly smaller than $\pi$ (see Figure~\ref{fig:jordan}).

        \begin{figure}[h!]
        \centering
        \begin{tikzpicture}[line width=1.5]
  \def\gridW{7cm} \def\gridH{5cm}
      \colorlet{mygreen}{green!70!black}

  \begin{scope}[shift={(0,0)}]
    \fill[magenta,opacity=0.5]  (-0.55,1.04) -- (-1,1.7) arc(120:215:0.7) --cycle;
    \draw[red] (-2,0) -- (2,0) node[right, red] {$\ell_{j}$};
    \draw[red] (-2, 3) -- (2, 1) node[right, red] {$\ell_{k}$};

    \draw[mygreen, thick, ->, line width=1.5] plot[smooth,tension=1] coordinates{(1.5, 1.25)  (0,0.6) (-1.5,2.75)};
    \draw[mygreen, thick, ->, line width=1.5] plot[smooth,tension=1] coordinates{(-1.7, 0) (0,1.2) (1.7,0)};

    \draw[blue, line width=4] plot[smooth, tension=1] coordinates {(-1.2,0.6) (-0.8,0.95) (-0.55,1.04)};

    \draw[blue, -latex, line width=4] plot[smooth, tension=1] coordinates {(-0.55,1.04) (-0.8,1.4) (-1,1.7)};

    \node[mygreen, above right] at (-1.7,2.7) {\large $g_{k}^{-1}(-t)$};
    \node[mygreen, below] at (-2,0) {\large $g_{j}^{-1}(-t)$};
    \fill[orange] (-0.55,1.04) circle (4pt); 
  \end{scope}

  \begin{scope}[shift={(\gridW,0)}]
    \fill[magenta,opacity=0.5]  (-0.37,0.95)  -- (-1.25,1.35) arc(160:215:0.85) --cycle;
    \draw[red] (-2,0) -- (2,0) node[right, red] {$\ell_{j}$};
    \draw[red] (-2, 1.5) -- (2, 2.5) node[right, red] {$\ell_{k}$};

    \draw[mygreen, thick, ->, line width=1.5] plot[smooth,tension=1] coordinates{(1.5, 2.35)  (0,1) (-1.5,1.65)};
    \draw[mygreen, thick, ->, line width=1.5] plot[smooth,tension=1] coordinates{(-1.75, 0) (0,1) (1.75,0)};

    \draw[blue, line width=4] plot[smooth, tension=1] coordinates {(-1.2,0.5) (-0.8,0.85) (-0.37,0.95)};

    \draw[blue, -latex, line width=4] (-0.37,0.95)  .. controls (-1, 1.1) .. (-1.3,1.4);

    \node[mygreen, above right] at (-2,1.7) {\large $g_{k}^{-1}(-t)$};
    \node[mygreen, below] at (-2,0) {\large $g_{j}^{-1}(-t)$};
    \fill[orange] (-0.37,0.95) circle (4pt); 
  \end{scope}

  \begin{scope}[shift={(0, -\gridH)}]
        \draw[red] (-2,0) -- (2,0) node[right, red] {$\ell_{j}$};
    \draw[red] (-2, 0) -- (1.7, 2.5) node[right, red] {$\ell_{k}$};

    \draw[mygreen, thick, ->, line width=1.5] plot[smooth,tension=1] coordinates{(1.2, 2.15)  (0.6,1) (-0.6,0)};
    \draw[mygreen, thick, ->, line width=1.5] plot[smooth,tension=1] coordinates{(-1, 0.65) (0.6,1) (1.75,0)};

    \node[mygreen, above right] at (-2.2,0.7) {\large $g_{j}^{-1}(-t)$};
    \node[mygreen, below] at (-0.8,0) {\large $g_{k}^{-1}(-t)$};
    \node[blue, right] at (1,1.4) {\large \textbf{No Solution}};
    \fill[orange] (0.57,0.97) circle (4pt);
  \end{scope}

  \begin{scope}[shift={(\gridW, -\gridH)}]
      \fill[magenta,opacity=0.5]  (-0.67,0.67) -- (-1.2,1.3) arc(160:200:1.5) --cycle;
    \draw[red] (2,0) -- (-2,0) node[left, red] {$\ell_{j}$};
    \draw[red] (2, 0) -- (-1.7, 2.5) node[left, red] {$\ell_{k}$};

    \draw[mygreen, thick, ->, line width=1.5] plot[smooth,tension=1] coordinates{(0.2, 0)  (-1,1) (-1.3,2.2)};
    \draw[mygreen, thick, ->, line width=1.5] plot[smooth,tension=1] coordinates{(-1.4, 0) (-0.3,0.8) (1.2,0.5)};

    \draw[blue, line width=4] plot[smooth, tension=1] coordinates {(-1.2,0.2) (-1,0.5) (-0.67,0.7)};

    \draw[blue, -latex, line width=4] (-0.67,0.67)  .. controls (-1, 1) .. (-1.27,1.45);

    \node[mygreen, above right] at (-1.6,2.2) {\large $g_{j}^{-1}(-t)$};
    \node[mygreen, below] at (-1.5,0) {\large $g_{k}^{-1}(-t)$};
    \fill[orange] (-0.67,0.67) circle (4pt);
  \end{scope}

\end{tikzpicture}
        \caption{The unique solution for a curve going from $g_{j}^{-1}(-t)$ to $g_{k}^{-1}(-t)$ at $\beta_{j,k}(t)$ (highlighted in orange), assuming $\nabla g_j$ and $\nabla g_k$ point to the right of the curve and that the left turn angle is less than $\pi$, is $J^\beta_{j,k}(t)$.}
        \label{fig:jordan}
        \end{figure}
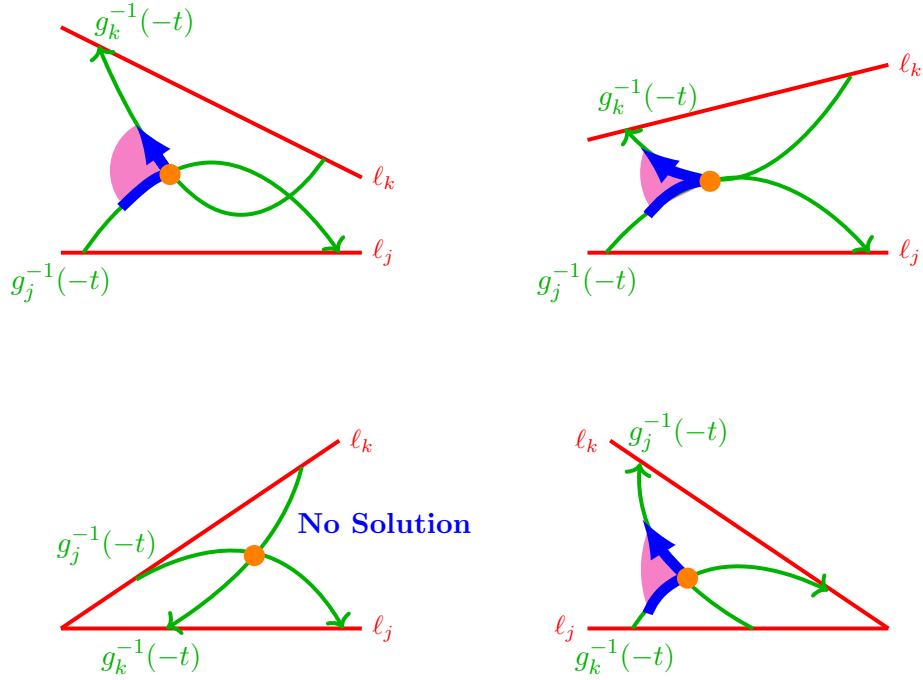
\end{proof}

\begin{lemma}\label{inward}
    Under the strict descent condition, given a regular loop $(t_0,\bm{\beta},\bm{g})$, there exists $\varepsilon>0$ such that for all $t\in (t_0,t_0+\varepsilon)$, the curve $(t,\bm{\beta}, \bm{g})$ is a regular loop inside $(t_0,\bm{\beta},\bm{g})$.
\end{lemma}

\begin{proof}
    Suppose $(t_0,\bm{\beta},\bm{g})$ is a regular loop. Then Condition~$\eqref{acute}$ would imply in particular that $J^{\beta}_{i,j}(t_0)$ is a well-defined curve with positive length for all $\beta_{i,j}\in \bm{\beta}$. Therefore, $\beta_{i,j}$ should be well defined for all $(t_0,t_0+\varepsilon)$ and all $\beta_{i,j}\in \bm{\beta}$ given $\varepsilon>0$ sufficiently small. Suppose $\varepsilon>0$ is small enough that for all $\beta_{i,j}\in \bm{\beta}$, $\beta_{i,j}(t)$ is well-defined for $t\in (t_0,t_0+\varepsilon)$. On $[t_0,t_0+\varepsilon]$, the loop $(t,\bm{\beta},\bm{g})$ evolves continuously. Thus we can choose a smaller $\varepsilon$ such that on $(t_0,t_0+\varepsilon)$, each arc of $(t,\bm{\beta},\bm{g})$ has positive length and non-adjacent arcs remain positive distance away from each other. In consequence, $(t,\bm{\beta},\bm{g})$ is still a Jordan curve with the same counter-clockwise orientation. Note that on each $(t,\beta_{i,j},g_j,\beta_{j,k})$, $\nabla g_j$ never vanishes and always points to one side of the curve by convexity. Thus Condition~$\eqref{concave}$ should also hold for all $t \in (t_0,t_0+\varepsilon)$. In summary, given $\varepsilon>0$ sufficiently small, Conditions~$\eqref{counter}$, $\eqref{pos}$ and $\eqref{concave}$ hold for all $(t,\bm{\beta}, \bm{g})$ with $t\in (t_0,t_0+\varepsilon)$.

    Consider a pair of adjacent arcs $(t_0,\beta_{i,j},g_j,\beta_{j,k})$ and $(t_0, \beta_{j,k}, g_k,\beta_{k,l})$. Note that the strict descent condition implies that $\beta_{i,j}((t_0,t_0+\varepsilon])$ does not intersect $g_j^{-1}(-t_0)$. The fact that $(t, \beta_{i,j},g_j,\beta_{j,k})$ is non-degenerate for $t\in (t_0,t_0+\varepsilon)$ implies that $\beta_{i,j}([t_0,t_0+\varepsilon])$ does not intersect the entire image of $\beta_{j,k}$. Thus it follows that $\beta_{i,j}([t_0,t_0+\varepsilon])$ is entirely contained in the component of $U^\beta_{j,k}(t_0) \backslash S_{j,k}$ that touches $\ell_j$. Likewise, $\beta_{k,l}([t_0,t_0+\varepsilon])$ is entirely contained in the component of $U^\beta_{j,k}(t_0) \backslash S_{j,k}$ that touches $\ell_k$. Therefore Corollary~\ref{geo_level} would imply for each $t\in [t_0,t_0+\varepsilon]$, both $(t,\beta_{i,j},g_j,\beta_{j,k})$ and $(t, \beta_{j,k}, g_k,\beta_{k,l})$ are in $J^\beta_{j,k}(t)$. Thus Condition~$\eqref{acute}$ holds for all $t\in[t_0,t_0+\varepsilon]$ by Corollary~\ref{equiv}.

    Now we claim that $(t,\bm{\beta},\bm{g})$ is disjoint from $(t_0,\bm{\beta},\bm{g})$ for all $t\in (t_0,t_0+\varepsilon)$ with sufficiently small $\varepsilon>0$. It suffices to show that for each $(t_0,\beta_{i,j},g_j,\beta_{j,k}) \subseteq (t_0,\bm{\beta},\bm{g})$, $(t,\beta_{i,j},g_j,\beta_{j,k})$ and  $(t_0,\bm{\beta},\bm{g})$ are disjoint for all $t\in (t_0,t_0+\varepsilon)$ with sufficiently small $\varepsilon>0$. Note that if $(t_0,\beta_{i,j},g_j,\beta_{j,k})$ and $(t_0,\beta_{i',j'},g_{j'},\beta_{j',k'})$ are disjoint, then $(t,\beta_{i,j},g_j,\beta_{j,k})$ and $(t_0,\beta_{i',j'},g_{j'},\beta_{j',k'})$ are also disjoint given $t-t_0$ sufficiently small. Thus we only need to consider the two arcs that share a vertex with $(t_0, \beta_{i,j},g_j,\beta_{j,k})$. Suppose $(t_0,\beta_{j,k},g_k,\beta_{k,l})$ is the one that shares vertex $\beta_{j,k}(t)$ with it. Thus by Condition~$\eqref{acute}$, $(t, \beta_{i,j},g_j,\beta_{j,k})$ and  $(t_0, \beta_{j,k},g_k,\beta_{k,l})$ lie in the different components of $W_{j,k}\backslash S_{j,k}$, implying that they are disjoint. Therefore, given $\varepsilon>0$ sufficiently small, for all $t\in (t_0,t_0+\varepsilon)$, $(t,\beta_{i,j},g_j,\beta_{j,k})$ is disjoint from all the arcs in $(t_0,\bm{\beta},\bm{g})$, proving the claim. Since the angle $J^\beta_{i,j}(t_0)$ makes toward $\beta_{i,j}([t_0,t_0+\varepsilon])$ is the internal angle of $(t_0,\bm{\beta},\bm{g})$, it follows that $\beta_{i,j}(t)$ is inside $(t_0,\bm{\beta},\bm{g})$ for $t\in (t_0,t_0+\varepsilon)$. Thus $(t,\bm{\beta},\bm{g})$ is inside $(t_0,\bm{\beta},\bm{g})$, proving the lemma.
\end{proof}

\subsection{Critical loops}
Given a regular loop $(t_0,\bm{\beta}, \bm{g})$, it is clear that $\beta_{i,j}(t)$ is well-defined on some half-closed interval $(t_0,t_0+\varepsilon]$ for all $\beta_{i,j}\in\bm{\beta}$. Thus $(t,\bm{\beta},\bm{g})$ is a well-defined (not necessarily simple) loop on such a interval. We set \begin{equation*}
    t_c = \inf\set{t\in (t_0,\infty)| \text{$(t,\bm{\beta},\bm{g})$ is not well-defined or it is not a regular loop}}.
\end{equation*}
By Lemma~\ref{inward}, $t_c$ is strictly greater than $t_0$. The finiteness of $t_c$ and well-defined-ness of $(t_c,\bm{\beta},\bm{g})$ are guaranteed by the following:

\begin{lemma}
    Assuming the condition of strict descent and that $(t_0,\bm{\beta},\bm{g})$ is a regular loop, we have $t_c < \infty$ and $(t_c,\bm{\beta},\bm{g})$ is well-defined. Additionally, $t\mapsto (t,\bm{\beta},\bm{g})$ is a continuously shrinking family of Jordan curves on $[t_0,t_c)$ and either of the following conditions needs to hold for the loop $(t_c,\bm{\beta},\bm{g})$:
\begin{enumerate}[(a)]
    \item Some of $(t_c,\beta_{i,j},g_j,\beta_{j,k})$'s are degenerate to points, or \label{lem_con1}
    \item $(t_c,\bm{\beta},\bm{g})$ is not a Jordan curve. \label{lem_con2}
\end{enumerate}
\end{lemma}
\begin{proof}
    Note that for all $t \in [t_0,t_c)$, any non adjacent arcs of $(t,\bm{\beta},\bm{g})$ are disjoint and all arcs have strictly positive length. The Jordan curve theorem implies that such a loop $(t,\bm{\beta},\bm{g})$ is in fact a Jordan curve. According to Lemma~\ref{inward}, the function \begin{equation*}
        t \mapsto d((t_0,\bm{\beta}, \bm{g}), (t,\bm{\beta}, \bm{g}))
    \end{equation*} is a locally increasing function as $(t,\bm{\beta}, \bm{g})$ shrinks locally over $t$. It follows from continuity that this function increases globally and therefore $(t_1,\bm{\beta}, \bm{g})$ is always inside $(t_2,\bm{\beta}, \bm{g})$ for $t_1>t_2$ in $[t_0,t_c)$. Since inside $(t_0,\bm{\beta},\bm{g})$, each $g_i$ is lower-bounded, we cannot have $\beta_{i,j}(t)$ inside  $(t_0,\bm{\beta},\bm{g})$ for arbitrarily large $t$. This implies that $t_c<\infty$.
    
    Now suppose that Condition~$\eqref{pos}$ holds for $(t_c,\bm{\beta},\bm{g})$. Again on each $(t,\beta_{i,j},g_j,\beta_{j,k})$, $\nabla g_j$ never vanishes and always points to one side of the curve. Therefore, Condition~$\eqref{concave}$ should also hold for all $t \in [t_0,t_c]$ by the continuity of $\nabla g_i$'s. Moreover, using the same argument as in Lemma~\ref{inward}, we have that Condition~$\eqref{pos}$ holds for all $t\in [t_0,t_c]$ implies that Condition~$\eqref{acute}$ also holds for all $t\in[t_0,t_c]$. Therefore, if $(t_c,\bm{\beta},\bm{g})$ is a Jordan curve, $\bm{g}$ would induce a counterclockwise parametrization and it would thus imply that it is a regular loop. This would contradiction the definition of $t_c$.
    
    In conclusion, Conditions~$\eqref{concave}$ and $\eqref{acute}$ will not fall before non-adjacent arcs of $(t,\bm{\beta},\bm{g})$ touch or some of the arcs become degenerate.
\end{proof}

Given a regular loop $(t_0,\bm{\beta},\bm{g})$, we call the loop $(t_c,\bm{\beta},\bm{g})$ a \textit{critical loop} originated from $(t_0,\bm{\beta},\bm{g})$. The important fact the enables us to explicitly build the electrostatic skeleton is as follows:

\begin{proposition}\label{breakdown}
    Under the strict descent condition, given a regular loop $(t_0,\bm{\beta},\bm{g})$, the critical loop $(t_c,\bm{\beta},\bm{g})$ is either a point or a union of regular loops intersecting each other on their vertices. 
\end{proposition}

To prove this proposition, we will need the following lemma to deal with the degenerate arcs in $(t_c,\bm{\beta},\bm{g})$.

\begin{lemma}\label{nondeg}
    Under the strict descent condition, if we delete the degenerate arcs in $(t_c,\bm{\beta},\bm{g})$ to form a new tuple $(t_c,\bm{\beta}_c, \bm{g}_c)$ (which represents the same loop), then  $(t_c,\bm{\beta}_c, \bm{g}_c)$ satisfies Condition~$\eqref{acute}$ in Definition~\ref{reg_loop}.
\end{lemma}
\begin{proof}
    Let $\beta_0,\beta_1,\dots,\beta_{K} \in \bm{\beta}$ be such that $\beta_0(t), \beta_1(t),\dots, \beta_K(t)$ are vertices of $(t_c,\bm{\beta},\bm{g})$ ordered counter-clockwise and \begin{equation*}
        \beta_0(t_c) \neq \beta_1(t_c) = \beta_2(t_c)=\cdots = \beta_{K-1}(t_c) \neq \beta_K(t_c).
    \end{equation*} Then there exists $\delta>0$ such that the lengths of the arc between $\beta_0(t)$ and $\beta_1(t)$ and the arc between $\beta_{K-1}(t)$ and $\beta_K(t)$ are lower-bounded by $2\delta$ for $t\in [t_0,t_c]$. We let $\alpha_0(t) = \alpha_0(t,\delta)$ be the point on the arc between $\beta_0(t)$ and $\beta_1(t)$ that is $\delta$ away from $\beta_1(t)$ in arc length and set $\alpha_K(t) = \alpha_K(t,\delta)$ be the point between $\beta_{K-1}(t)$ and $\beta_{K}(t)$ that is $\delta$ away from $\beta_{K-1}(t)$ in arc length (see Figure~\ref{fig:path}). Then by convexity, we have \begin{equation*}
        |\alpha_0(t) - \beta_1(t)| < \delta,\quad \text{and}\quad |\beta_{K-1}(t) - \alpha_K(t)| <\delta \quad \text{for all $t\in [t_0,t_c]$}.
    \end{equation*}
    Now due to the orientation of the regular loops, for any $t\in [t_0,t_c)$, the curve segment $\mathcal{C}_t$ of $(t,\bm{\beta},\bm{g})$ that goes from $\alpha_0(t)$ to $\alpha_K(t)$ lies on the left of the piecewise linear path $\mathcal{L}_t = \mathcal{L}_t(\delta)$ joining $\alpha_0(t), \beta_1(t),\dots,\beta_{K-1}(t)$ and $\alpha_K(t)$. It follows that since $\mathcal{C}_t$ is a simple curve, $\mathcal{L}_t$ cannot turn left in the angle of $\pi$ or enclose a domain in a counter-clockwise manner. 
    \begin{figure}[h!]
        \centering
        \begin{tikzpicture}[>=stealth, line width=1.5pt]
    \colorlet{myred}{red!90!black}
    \colorlet{mygreen}{green!70!black}
    \colorlet{lightgreen}{green!40!white}

    \coordinate (start) at (-4, -1);
    \coordinate (p0)    at (-3.4, 0);
    \coordinate (p1)    at (-1.2, 1);
    \coordinate (p2)    at (0.0, 0.0);
    \coordinate (p3)    at (0.8, 0.8);
    \coordinate (p4)    at (2.1, 0.5);
    \coordinate (p5)    at (3.5, 2);
    \coordinate (end)   at (5.0, 0.7);

    \foreach \s/\e in {p0/p1, p1/p2, p2/p3, p3/p4, p4/p5} {
        \draw[myred, line width=1.8pt, 
              decoration={markings, mark=at position 0.6 with {\arrow{>}}},
              postaction={decorate}] (\s) -- (\e);
    }

    \draw[mygreen, line width=1.8pt] (start) .. controls (-3, 1) and (-2, 1.5) .. (p1);
    \draw[mygreen, line width=1.8pt] (p1) .. controls (-0.8, 1.4) and (-0.2, 0.8) .. (p2);
    \draw[mygreen, line width=1.8pt] (p2) .. controls (0.2, 1.2) and (0.6, 1) .. (p3);
    \draw[mygreen, line width=1.8pt] (p3) .. controls (1.2, 1.6) and (1.6, 1.4) .. (p4);
    \draw[->, mygreen, line width=1.8pt] (p4) .. controls (2.6, 2.4) and (3, 2.8) .. (end);

    \foreach \p in {p0, p1, p2, p3, p4, p5} {
        \fill[myred] (\p) circle (2.5pt);
    }

    \node[anchor=north west, color=myred] at (-3.7, 0) {\large $\alpha_0(t)$};
    \node[anchor=north east, color=myred] at (-0.8, 0.8) {\large $\beta_1(t)$};
    \node[anchor=north west, color=myred] at (-0.2, 0) {\large $\beta_2(t)$};
    \node[anchor=north west, color=myred] at (0.5, 0.6) {\large $\beta_3(t)$};
    \node[anchor=north west, color=myred] at (1.8, 0.5) {\large $\beta_4(t)$};
    \node[anchor=north west, color=myred] at (3, 1.8) {\large $\alpha_5(t)$};
    \node[anchor=north west, color=myred] at (3.5, 2.8) {\huge $\mathcal{L}_t$};

    \node[anchor=south east, color=mygreen] at (6, 0) {\large $\mathcal{C}_t\subseteq(t, \bm{\beta}, \bm{g})$};

\end{tikzpicture}
        \caption{$\mathcal{L}_t$ is highlighted in red and $(t,\bm{\beta},\bm{g})$ is highlighted in green}
        \label{fig:path}
    \end{figure}
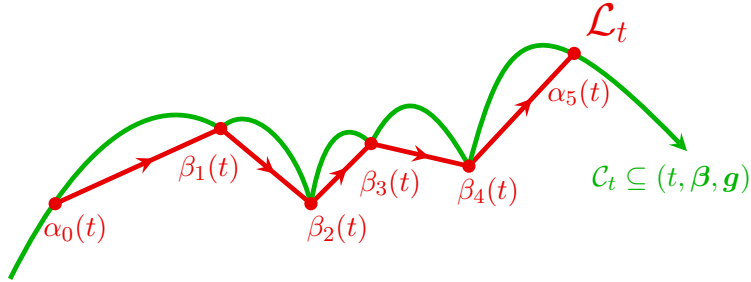
    However, we know that as $\mathcal{L}_t$ evolves continuously over $t$, \begin{enumerate}
        \item the first and last linear segments of $\mathcal{L}_t$ have lower-bounded length, and
        \item all other segments in between vanish.
    \end{enumerate}

    Now we note that the inside of $(t_0,\bm{\beta},\bm{g})$ is precompact, and that $\beta_{i,j}$'s are continuous and $\nabla g_i$'s are smooth. Therefore, by uniform continuity, there exist $\varepsilon_0, \delta_0>0$ such that for all $t\in [t_0,t_c]$, \begin{enumerate}
        \item The angle any $J^\beta_{i,j}(t)$ makes toward $U^\beta_{i,j}(t)$ is upper bounded by $\pi - 2\varepsilon_0$;
        \item For any $J^\beta_{i,j}(t)$ and $z_1, z_2\in J^\beta_{i,j}(t)$ separated by $\beta_{i,j}(t)$, $|z_1 -z_2| \leq 2\delta_0$ implies that the angle of $[z_1,\beta_{i,j}(t)]\cup [\beta_{i,j}(t),z_2]$ is $\varepsilon_0$ away from the angle of $J^\beta_{i,j}(t)$. In particular, the angle of $[z_1,\beta_{i,j}(t)]\cup [\beta_{i,j}(t),z_2]$ would be upper bounded by $\pi - \varepsilon_0$.
    \end{enumerate}

    Now we suppose $t$ is sufficiently close to $t_c$ such that $|\beta_k(t)-\beta_{k+1}(t)| < \delta_0$ for all $k=1,\dots,K-2$. This implies that the sum of left turns for $\mathcal{L}_t$ from $\beta_2(t)$ to $\beta_{K-2}(t)$ is at least $(K-3)\varepsilon_0\geq0$. In particular, this means that the sum of left turns at $\beta_1(t)$ and $\beta_K(t)$ is upper bounded by $\pi -(K-3) \varepsilon_0$, otherwise $\mathcal{L}_t$ will enclose the domain counter-clockwise. Therefore, we can conclude that at $\beta_1(t_c) = \cdots = \beta_{K-1}(t_c)$,  $\mathcal{L}_{t_c}$ cannot make the left turn for more than $\pi -(K-3)\varepsilon_0$.
    
    On the other hand, if we set $\delta$ in the definition of $\alpha_0$ and $\alpha_K$ to be less than $\delta_0$, then the left turns of $\mathcal{L}_t$ at $\beta_1(t)$ and $\beta_{K-1}(t)$ are also lower-bounded by $\varepsilon_0>0$ given $t$ sufficiently close to $t_c$. As a result, the left turn of $\mathcal{L}_{t_c}$ at $\beta_1(t_c) = \cdots =\beta_{K-1}(t_c)$ is at least $(K-1)\varepsilon_0$. 

    In conclusion, the angle $\mathcal{L}_{t_c}$ makes at $\beta_1(t_c)=\cdots = \beta_{K-1}(t_c)$ (equal to $\pi$ minus the left turn angle) is between $(K-3)\varepsilon_0$ and $\pi - (K-1)\varepsilon_0$. Thus the internal angle of $(t_c,\bm{\beta},\bm{g})$, which is upper-bounded by the angle of $\mathcal{L}_{t_c}$, is less than $\pi - (K-1)\varepsilon_0$.
    
    Thus, we rewrite $(t_c,\bm{\beta},\bm{g})$ as $(t_c,\bm{\beta}_c,\bm{g}_c)$ by removing degenerate arcs. The new tuple $(t_c,\bm{\beta}_c,\bm{g}_c)$ would satisfy Condition~$\eqref{acute}$.
\end{proof}

\begin{proof}[Proof of Proposition~\ref{breakdown}]
    Assume that $(t_c,\bm{\beta},\bm{g})$ is not a point. Note that since $\set{(t,\bm{\beta},\bm{g})}_{t\in[t_0,t_c)}$ is a family of continuously shrinking Jordan curves, Carathéodory kernel theorem implies that $(t_c,\bm{\beta},\bm{g})$ has a simply connected outside on $\hat{\C}$ and that its inside, which combined with the loop, is also simply connected. In particular, since $\nabla  g_i$'s are continuous inside the polygon, $\nabla g_j$ should still point toward the outside of $(t_c,\bm{\beta},\bm{g})$ on any non-degenerate $(t_c,\beta_{i,j},g_j,\beta_{j,k})$ and $\bm{g}$ is still ordered in a counter-clockwise fashion. As a consequence, on each non-degenerate arc $(t_c,\beta_{i,j},g_j,\beta_{j,k})$, going from $\beta_{i,j}(t_c)$ to $\beta_{j,k}(t_c)$, $\nabla g_j$ points to the right of the curve. Therefore, by Lemma~\ref{nondeg}, we can re-write $(t_c,\bm{\beta},\bm{g})$ as a new loop $(t_c,\bm{\beta}_c,\bm{g}_c)$ in such a way that none of its arcs are degenerate and Conditions~$\eqref{acute}$ holds. Then if the new loop is a Jordan curve, it has to be regular per Definition~\ref{reg_loop}.
    
    Now suppose that $(t_c,\bm{\beta}_c,\bm{g}_c)$ is not a Jordan curve. Carathéodory kernel theorem then implies that there exists a conformal map $f:\D \to \hat{\C}$ that is continuous on $\overline{\D}$ such that \begin{equation*}
        f(\partial \D) = (t_c,\bm{\beta}_c,\bm{g}_c),\quad f(\D) = \bigcup_{t\in [t_0,t_c)} \text{exterior of $(t,\bm{\beta},\bm{g})$}.
    \end{equation*}
    Then $\theta \mapsto f(e^{-i\theta})$ is a counter-clockwise parametrization of $(t_c,\bm{\beta}_c,\bm{g}_c)$. Since each arc of $(t_c,\bm{\beta}_c,\bm{g}_c)$ only intersects each other finitely many times, the equivalence relation induced by $f$ on $\partial \D$ has only finitely many non-singleton equivalence classes $\set{[z_1], [z_2],\dots,[z_n]}$. This equivalance relation is non-crossing. This implies that $(t_c,\bm{\beta}_c,\bm{g}_c)$ along with its inside is simply connected and consists of finitely many Jordan domains intersecting each other only at $f([z_i])$. Note that given $t\in [t_0,t_c)$, the external angle of the loop $(t,\bm{\beta},\bm{g})$ going counter-clockwise is $\pi$ almost everywhere and always not less than $\pi$ according to Lemma~\ref{cuts}. By the continuity of $\beta_{i,j}$'s and $\nabla g_i$'s, this also holds when $t=t_c$. Therefore, at each $f([z_i]) \in (t_c,\bm{\beta}_c,\bm{g}_c)$ where the curve has multiple angles facing outside of $(t_c,\bm{\beta}_c,\bm{g}_c)$, the externals angle can only be $\pi$ and thus $|[z_i]|=2$. Therefore, if two Jordan curves in $(t_c,\bm{\beta}_c,\bm{g}_c)$ intersect at $a$, the internal angles of both at $a$ should be $0$ and the counter-clockwise curve parametrizing either of them will turn left in an angle of $\pi$. 

    In conclusion, for each Jordan curve in $(t_c,\bm{\beta}_c,\bm{g}_c)$, \begin{enumerate}
        \item when parametrized counter-clockwise, the curve turns left for less than $\pi$ at $\beta_{i,j}(t_c)$ for $\beta_{i,j}\in \bm{\beta}_c$ and turns right everywhere else, and
        \item at the point where the curve intersects another Jordan curve, it turns left in the angle of $\pi$.
    \end{enumerate} 
    Since each Jordan curve consists of segments of level sets $g_i^{-1}(-t_c)$'s, this implies that they are regular loops and that the points where they intersect are vertices.
\end{proof}

From the proof of Proposition~\ref{breakdown}, we can also infer the following important property of the decomposition of a critical loop:

\begin{corollary}[Non-crossing matching]\label{non-cross}
    Let $(t_c,\bm{\beta}_1,\bm{g}_1),\dots,(t_c,\bm{\beta}_l,\bm{g}_l)$ be the set of regular loops that make up the critical loop $(t_c,\bm{\beta},\bm{g})$. Suppose each element of $\bm{g} \in \set{g_1,\dots,g_n}^m$ is distinct and we identify each element in $\bm{g}$ with a vertex in a convex $m$-gon $P$ in a counter-clockwise manner. Then each tuple $\bm{g}_i$ represents a sub-polygon $P_i$ inside $P$. Moreover, $P_1,\dots, P_l$ have disjoint interior (see Figure~\ref{fig:non-cross}).
\end{corollary}

    \begin{figure}[ht]
        \centering
        \begin{tikzpicture}[scale=1.2, line width=1.7pt] 
        \begin{scope}
            \fill[orange, opacity=0.4] plot[smooth, tension=1] coordinates {(-0.8,0.25) (-0.6,0.1)(-0.5,-0.2)} --  plot[smooth, tension=1] coordinates {(-0.5,-0.2) (-0.1,0.05) (0.3,0.1)} -- plot[smooth, tension=1] coordinates {(0.3,0.1) (0.1, 0.5) (0,1)} -- plot[smooth, tension=1] coordinates {(0,1) (-0.2,0.4) (-0.8,0.25)} -- cycle;

            \fill[magenta, opacity=0.4] plot[smooth, tension=1] coordinates {(0.3,0.1) (0.6,-0.1) (0.8,-0.4)}-- plot[smooth, tension=1] coordinates {(0.8,-0.4) (0.9,-0.2) (1.2,0)} -- plot[smooth, tension=1] coordinates {(1.2,0) (0.75, -.03) (0.3,0.1)} -- cycle;
            
            \draw[blue] plot[smooth, tension=1] coordinates {(-1,-1) (-0.8,-0.6) (-0.6, -0.3) (-0.5,-0.2)};
            \draw[blue] plot[smooth, tension=1] coordinates {(1,-1) (0.9,-0.7) (0.8,-0.4)};
            \draw[blue] plot[smooth, tension=1] coordinates {(2,0) (1.6,0.05) (1.2,0)};
            \draw[blue] plot[smooth, tension=1] coordinates {(0.8,1.8) (0.6,1.6) (0.4,1.2) (0,1)};
            \draw[blue] plot[smooth, tension=1] coordinates {(0,2) (0.1,1.5) (0,1)};
            \draw[blue] plot[smooth, tension=1] coordinates {(-1.6,0.5) (-1.3,0.4) (-1,0.4) (-0.8,0.25)};

            \draw plot[smooth, tension=1] coordinates {(-1,-1) (0,-0.5) (1,-1)};
            \draw plot[smooth, tension=1] coordinates {(1,-1) (1.2, -0.3) (2,0)}; 
            \draw plot[smooth, tension=1] coordinates {(2,0) (1.1,0.6) (0.8,1.8)};
            \draw plot[smooth, tension=1] coordinates {(0.8,1.8) (0.4,1.7) (0,2)};
            \draw plot[smooth, tension=1] coordinates {(0,2) (-0.6,1) (-1.6,0.5)};
            \draw plot[smooth, tension=1] coordinates {(-1.6,0.5) (-1.1,0) (-1,-1)};
            
            \draw[teal] plot[smooth, tension=1] coordinates {(-0.5,-0.2) (0.3,0.1) (0.8,-0.4)};
            \draw[teal] plot[smooth, tension=1] coordinates {(0.8,-0.4) (0.9,-0.2) (1.2,0)}; 
            \draw[teal] plot[smooth, tension=1] coordinates {(1.2,0) (0.3,0.15) (0,1)};
            \draw[teal] plot[smooth, tension=1] coordinates {(0,1) (-0.2,0.4) (-0.8,0.25)};
            \draw[teal] plot[smooth, tension=1] coordinates {(-0.8,0.25) (-0.6,0.1)(-0.5,-0.2)};
            
            \draw[-stealth] (-1.4,0) -- (-0.6,0);
            \node[teal, left] at (-1.4,0) {\large $(t_c,\bm{\beta},\bm{g})$};
            \draw[-stealth] (-0.5,2) -- (-0.5,1.2);
            \node[above] at (-0.5,2) {\large $(t_0,\bm{\beta},\bm{g})$};
            \node at (-1.3,-0.4) {$\bm{g}[5]$};
            \node at (0,-0.8) {$\bm{g}[6]$};
            \node at (1.4,-0.5) {$\bm{g}[1]$};
            \node at (1.5,0.7) {$\bm{g}[2]$};
            \node at (0.4,2.1) {$\bm{g}[3]$};
            \node at (-1,1.1) {$\bm{g}[4]$};
        \end{scope}

        \begin{scope}[shift={(4.5,0.5)}, scale=0.8]
            \fill[orange, opacity=0.4] ({2*cos(30)},{2*sin(30)}) -- ({2*cos(150)},{2*sin(150)}) -- ({2*cos(210)},{2*sin(210)}) -- ({2*cos(270)},{2*sin(270)}) -- cycle;
            \fill[magenta, opacity=0.4] ({2*cos(270)},{2*sin(270)}) -- ({2*cos(330)},{2*sin(330)}) -- ({2*cos(30)},{2*sin(30)}) -- cycle;
            
            \draw ({2*cos(30)},{2*sin(30)}) -- ({2*cos(90)},{2*sin(90)}) -- ({2*cos(150)},{2*sin(150)}) -- ({2*cos(210)},{2*sin(210)}) -- ({2*cos(270)},{2*sin(270)}) -- ({2*cos(330)},{2*sin(330)}) --cycle;

            \node[right] at ({2*cos(30)},{2*sin(30)}) {$\bm{g}[2]$};
            \node[above] at ({2*cos(90)},{2*sin(90)}) {$\bm{g}[3]$};
            \node[left] at ({2*cos(150)},{2*sin(150)}) {$\bm{g}[4]$};
            \node[left] at ({2*cos(210)},{2*sin(210)}) {$\bm{g}[5]$};
            \node[below] at ({2*cos(270)},{2*sin(270)}) {$\bm{g}[6]$};
            \node[right] at ({2*cos(330)},{2*sin(330)}) {$\bm{g}[1]$};

            \draw ({2*cos(30)},{2*sin(30)}) -- ({2*cos(150)},{2*sin(150)});
            \draw ({2*cos(30)},{2*sin(30)}) -- ({2*cos(270)},{2*sin(270)});

        \end{scope}
        \end{tikzpicture}
        \caption{Identifying the phase transition with a partition of polygon. Note that $\bm{g}[i]$ stands for the $i$-component of $\bm{g}$ (not to be confused with $g_i$). Each regular loop in $(t_c,\bm{\beta},\bm{g})$ corresponds to a sub-polygon in the partition (color matched).}
        \label{fig:non-cross}
    \end{figure}
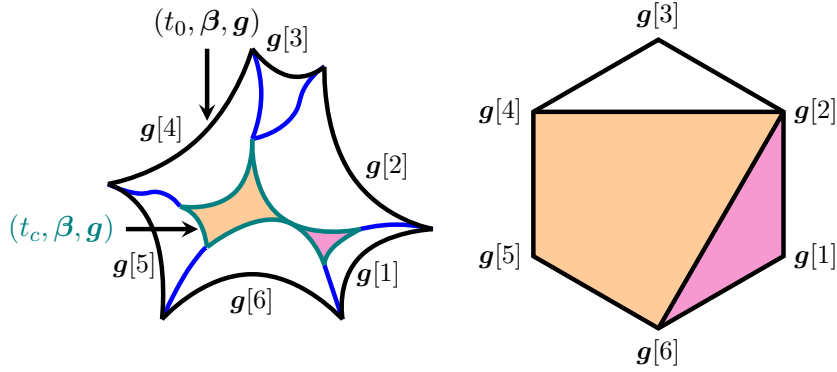

\begin{proof}
    Note that we have a conformal map $f:\D \to \hat{\C}$ continuous on $\overline{\D}$ such that \begin{align*}
        &f(\partial \D) = (t_c,\bm{\beta},\bm{g}),\quad
        f(\D) =  \text{exterior of $(t_c,\bm{\beta},\bm{g})$},\\
        &|f^{-1}(z)| \in \set{1,2},\quad \text{for all $z\in (t_c,\bm{\beta},\bm{g})$}.
    \end{align*}
    Additionally, the equivalence relation on $\partial \D$ induced by $f$ is non-crossing. Since each element in $\bm{g}$ is unique, we can denote their corresponding arcs as $C_1,\dots,C_m$ in order. Let $F: [0,2\pi]$ be a reparametrization of $\theta \mapsto f(e^{-i\theta})$ such that \begin{equation*}
        F(\left[\frac{2\pi(k-1)}{m}, \frac{2\pi k}{m}\right]) = C_k.
    \end{equation*} Note that since some of  $C_k$'s might be degenerate, this reparametrization is not reversible. That is, there exists a increasing function $r$ such that \begin{equation*}
        F(\theta) = f(e^{-ir(\theta)}),
    \end{equation*} but $r$ might not be bijective. Nevertheless, since the orientation is preserved, the equivalence relation on $[0,2\pi)$ induced by $F$ is still non-crossing. Note that this means for $a_1,a_2,b_1,b_2 \in [0,2\pi)$ if \begin{equation*}
        a_1 \sim a_2,\quad b_1\sim b_2,\quad a_1<a_2,\quad b_1<b_2,\quad a_1\leq b_1
    \end{equation*} then $a_1\leq b_1<b_2 \leq a_2$. Now let $v: [0,2\pi) \to [0,2\pi)$ be a piecewise linear continuous function such that if $\theta \in [2\pi(k-1)/m, 2\pi k/m]$ is equivalent to some $\theta' \neq \theta$, we have \begin{equation*}
        v(\theta) = 2\pi (k-1)/m.
    \end{equation*} Then non-crossing implies that $v$ is increasing. Let $\approx$ be the equivalence relation on $[0,2\pi)$ induced by $F\circ v$. Then $\approx$ is a non-crossing equivalence relation on $\set{2\pi k/m}_{k=0}^{m-1}$ such that $2\pi k_1/m  \approx  2\pi k_2/m$ if and only if $C_{k_1+1}$ and $C_{k_2+1}$ are two non-degenerate arcs that intersect in their interiors. Now for $k_1<k_2$ we also set $2\pi k_1/m \approx 2\pi k_2/m$, if $C_{k}$ is degenerate for all $k=k_1+2,k_1+2,\dots, k_2$. Since none of the degenerate arcs can intersect other arcs in the interior, the equivalence relation is still non-crossing.
    
    In summary, we have identified each $C_k$ with $2\pi(k-1)/m$. Note that since $C_k$'s follows the order of $\bm{g}$, we label the arcs in the counter-clockwise manner as $k$ increases. For $k_1<k_2$, $C_{k_1}$ is related to $C_{k_2}$ under $\approx$ if they are both non-degenerate arcs and either \begin{enumerate}
        \item $C_{k_1}$ and $C_{k_2}$ intersect in the interiors of both arcs, or
        \item $C_k$ is degenerate for all $k=k_1+1,\dots ,k_2-1$.
    \end{enumerate}
    On the other hand, as a non-crossing matching on $\set{e^{2\pi k/m}}_{k=0}^{m-1}$, $\approx$ partitions the regular $m$-gon $\mathrm{Conv}\set{z|z^m=1}$ into sub-polygons. Then each regular loop in $(t_c,\bm{\beta},\bm{g})$ can be identified with a unique sub-polygon under the partition. The result then follows.
\end{proof}

Since a triangle does not admit a non-crossing matching, it follows that \begin{corollary}\label{tri}
    For a regular loop $(t_0,\bm{\beta},\bm{g})$ with $|\bm{g}| = 3$, $(t_c,\bm{\beta},\bm{g})$ is a single point.
\end{corollary}

\subsection{Building the skeleton}
\begin{definition}[Loop measure]\label{loop_measure}
    Given a loop $(t,\bm{\beta},\bm{g})$ inside a convex polygon $\Omega$, the loop measure $\mu$ of $(t,\bm{\beta},\bm{g})$ is the unique measure such that \begin{equation*}
        \mu (U) = \sum_{(t,\beta_{i,j},g_j,\beta_{j,k})\subseteq (t,\bm{\beta},\bm{g})} \frac{1}{2\pi}\int_{(t,\beta_{i,j},g_j,\beta_{j,k})\cap U} |\nabla g_j(\zeta)| |d\zeta|.
    \end{equation*}
\end{definition}
Note that since we are integrating along the level sets, we have \begin{equation*}
    |\nabla g_j(\zeta)| = \left|\pder{g_j}{\mathbf{n}}(\zeta)\right|, \quad \zeta \in (t,\beta_{i,j},g_j,\beta_{j,k}).
\end{equation*}  The equilibrium measure $\mu_E$ of the polygon is the loop measure of the regular loop \begin{equation*}
    \partial \Omega = (0,(\gamma_{1,2},\gamma_{2,3},\dots,\gamma_{n-1,n},\gamma_{1,n}),(g_2,g_3,\dots,g_n,g_1)).
\end{equation*}

\begin{lemma}\label{replace1}
    Given a regular loop $(t_0,\bm{\beta},\bm{g})$ with loop measure $\mu_0$ and its corresponding critical loop $(t_c,\bm{\beta},\bm{g})$ with loop measure $\mu_c$, if we set $\lambda$ to be a measure supported on the annulus between $(t_0,\bm{\beta},\bm{g})$ and $(t_c,\bm{\beta},\bm{g})$ such that \begin{equation*}
        \lambda(U) = \sum_{\beta_{i,j}\in \bm{\beta}} \int _{\beta_{i,j}([t_0,t_c]) \cap U} |\nabla g_i(\zeta) - \nabla g_j(\zeta)| |d\zeta|,
    \end{equation*} then $U^{\mu_0}$ and $U^{\mu_c+\lambda}$ coincide on the exterior of $(t_0,\bm{\beta},\bm{g})$. That is, $\mu_0$ is the \textit{balayage} of $\mu_c + \lambda$ on  $(t_0,\bm{\beta},\bm{g})$.
\end{lemma}

\begin{proof}
    Given an arc $(t_0,\beta_{i,j},g_j,\beta_{j,k})\subseteq (t_0,\bm{\beta},\bm{g})$, we set $SQ(t_0,\beta_{i,j},g_j,\beta_{j,k})$ (see Figure~\ref{fig:web}) to be the Jordan domain enclosed by \begin{equation*}
        (t_0,\beta_{i,j},g_j,\beta_{j,k}) \cup \beta_{i,j}([t_0,t_c]) \cup (t_c,\beta_{i,j},g_j,\beta_{j,k}) \cup \beta_{j,k}([t_0,t_c]).
    \end{equation*}

    \begin{figure}[h!]
        \centering
        \begin{tikzpicture}[scale=1.6, line width=1.7pt] 
            \fill[orange, opacity=0.5] plot[smooth, tension=1] coordinates {(2,0) (1.1,0.6) (0.8,1.8)} -- plot[smooth, tension=1] coordinates {(0.8,1.8) (0.6,1.6) (0.4,1.2) (0,1)}  --plot[smooth, tension=1] coordinates {(0,1) (0.3,0.15)  (1.2,0)}--plot[smooth, tension=1] coordinates {(1.2,0) (1.6,0.05) (2,0)}--cycle;

            \draw[blue] plot[smooth, tension=1] coordinates {(-1,-1) (-0.8,-0.6) (-0.6, -0.3) (-0.5,-0.2)};
            \draw[blue] plot[smooth, tension=1] coordinates {(1,-1) (0.9,-0.7) (0.8,-0.4)};
            \draw[blue] plot[smooth, tension=1] coordinates {(2,0) (1.6,0.05) (1.2,0)};
            \draw[blue] plot[smooth, tension=1] coordinates {(0.8,1.8) (0.6,1.6) (0.4,1.2) (0,1)};
            \draw[blue] plot[smooth, tension=1] coordinates {(0,2) (0.1,1.5) (0,1)};
            \draw[blue] plot[smooth, tension=1] coordinates {(-1.6,0.5) (-1.3,0.4) (-1,0.4) (-0.8,0.25)};

            \draw plot[smooth, tension=1] coordinates {(-1,-1) (0,-0.5) (1,-1)};
            \draw plot[smooth, tension=1] coordinates {(1,-1) (1.2, -0.3) (2,0)}; 
            \draw plot[smooth, tension=1] coordinates {(2,0) (1.1,0.6) (0.8,1.8)};
            \draw plot[smooth, tension=1] coordinates {(0.8,1.8) (0.4,1.7) (0,2)};
            \draw plot[smooth, tension=1] coordinates {(0,2) (-0.6,1) (-1.6,0.5)};
            \draw plot[smooth, tension=1] coordinates {(-1.6,0.5) (-1.1,0) (-1,-1)};
            
            \draw[teal] plot[smooth, tension=1] coordinates {(-0.5,-0.2) (0.3,0.1) (0.8,-0.4)};
            \draw[teal] plot[smooth, tension=1] coordinates {(0.8,-0.4) (0.9,-0.2) (1.2,0)}; 
            \draw[teal] plot[smooth, tension=1] coordinates {(1.2,0) (0.3,0.15) (0,1)};
            \draw[teal] plot[smooth, tension=1] coordinates {(0,1) (-0.2,0.4) (-0.8,0.25)};
            \draw[teal] plot[smooth, tension=1] coordinates {(-0.8,0.25) (-0.6,0.1)(-0.5,-0.2)};
            \node at (1.8,1.3) {\large $(t_0,\beta_{i,j},g_j,\beta_{j,k})$};
            \node[orange] at (2.6,0.5) { $SQ(t_0,\beta_{i,j},g_j,\beta_{j,k})$};
            \draw[-stealth] (1.6,0.5) -- (0.9,0.5);
            \draw[-stealth] (-1.4,0) -- (-0.6,0);
            \node[teal, left] at (-1.4,0) {\large $(t_c,\bm{\beta},\bm{g})$};
        \end{tikzpicture}
        \caption{Given an arc arc $(t_0,\beta_{i,j},g_j,\beta_{j,k})$ from a regular loop, the Jordan domain $SQ(t_0,\beta_{i,j},g_j,\beta_{j,k})$ is the analytic square (or triangle) highlighted in orange}
        \label{fig:web}
    \end{figure}
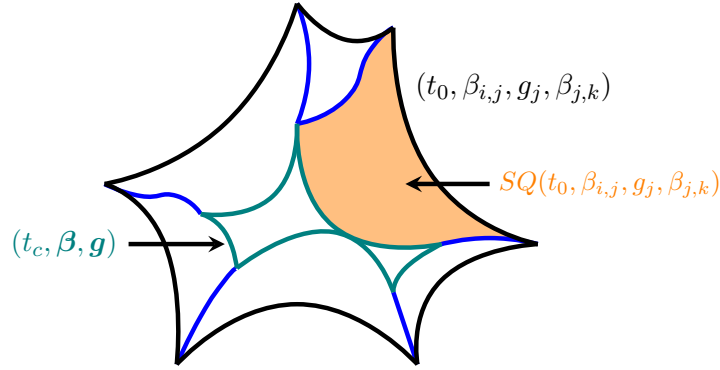
    
    Then the function \begin{equation*}
        u(z) = \begin{cases}
            -t_0,&\text{$z$ is outside $(t_0,\beta_{i,j},g_j,\beta_{j,k})$,}\\
            g_j(z), & z\in SQ(t_0,\bm{\beta},\bm{g}),\\
            -t_c,& \text{$z$ is inside $(t_c,\beta_{i,j},g_j,\beta_{j,k})$},
        \end{cases}
    \end{equation*}
    is continuous on $\C$ and harmonic away from \begin{equation*}
        \bigcup_{(t_0,\beta_{i,j},g_j,\beta_{j,k})\subseteq (t_0,\bm{\beta},\bm{g})} \partial SQ(t_0,\beta_{i,j},g_j,\beta_{j,k}).
    \end{equation*} By the Green-Gauss formula, for any $f\in C^\infty(\C)$, we have \begin{equation}\label{eq12}
        \int_\C f \Delta u = \sum_{ (t_0,\beta_{i,j},g_j,\beta_{j,k})\subseteq (t_0,\bm{\beta},\bm{g})} \int_{\partial SQ(t_0,\beta_{i,j},g_j,\beta_{j,k})} f(\zeta) \pder{g_j}{\mathbf{n}_{sq_j}}(\zeta) |d\zeta|,
    \end{equation} where $\bm{n}_{sq_j}$ is the normal vector on $\partial SQ(t_0,\beta_{i,j},g_j,\beta_{j,k})$ pointing outward. Now on $(t_0,\beta_{i,j},g_j,\beta_{j,k})$, $\nabla g_j$ is the normal vector pointing toward outside of $SQ(t_0,\beta_{i,j},g_j,\beta_{j,k})$ and on $(t_c,\beta_{i,j},g_j,\beta_{j,k})$, $\nabla g_j$ points inward, thus each summand in $\eqref{eq12}$ is equal to\begin{align*}
         \int_{(t_0,\beta_{i,j},g_j,\beta_{j,k})} f(\zeta)|\nabla g_j(\zeta)| |d\zeta| - \int_{(t_c,\beta_{i,j},g_j,\beta_{j,k})}f(\zeta) |\nabla g_j(\zeta)| |d\zeta|\\
         + \int_{\beta_{i,j}([t_0,t_c])}  f(\zeta)\pder{g_j}{\mathbf{n}_{sq_j}}(\zeta) |d\zeta|+\int_{\beta_{j,k}([t_0,t_c])} f(\zeta)\pder{g_j}{\mathbf{n}_{sq_j}}(\zeta) |d\zeta|.
    \end{align*}
    Note that we have \begin{equation*}
        \beta_{i,j}'(t) \cdot (\nabla g_i(\beta_{i,j}(t))- \nabla g_j(\beta_{i,j}(t))) = 0.
    \end{equation*} Thus the fact that the angle of $J^\beta_{i,j}(t)$ makes in the descending direction of $\beta_{i,j}$ is strictlt less than $\pi$ implies that on $\beta_{i,j}([t_0,t_c])$, \begin{equation*}
        \pder{g_i}{\mathbf{n}_{sq_i}}(\zeta) + \pder{g_j}{\mathbf{n}_{sq_j}}(\zeta) = -|\nabla g_i(\zeta) - \nabla g_j(\zeta)| \leq0.
    \end{equation*} In summary, we have \begin{align*}
        \int_\C f \Delta u =&  \sum_{ (t_0,\beta_{i,j},g_j,\beta_{j,k})\subseteq (t_0,\bm{\beta},\bm{g})}\int_{(t_0,\beta_{i,j},g_j,\beta_{j,k})} f(\zeta)|\nabla g_j(\zeta)| |d\zeta| \\
        &- \sum_{ (t_c,\beta_{i,j},g_j,\beta_{j,k})\subseteq (t_c,\bm{\beta},\bm{g})}\int_{(t_c,\beta_{i,j},g_j,\beta_{j,k})}f(\zeta) |\nabla g_j(\zeta)| |d\zeta|\\
         &-\sum_{\beta_{i,j}\in \bm{\beta}} \int_{\beta_{i,j}([t_0,t_c])} f(\zeta) |\nabla g_i(\zeta) - \nabla g_j(\zeta)||d\zeta|.
    \end{align*}
    
    It follows that we have \begin{equation*}
        \Delta u = \mu_0 - \mu_c -\lambda.
    \end{equation*} Since $u$ is constant outside $(t_0,\bm{\beta},\bm{g})$, so is $U^{\Delta u}$. Thus $U^{\Delta u}$ vanishes outside $(t_0,\bm{\beta},\bm{g})$ and the result then follows.
\end{proof}

Given a loop $(t,\bm{\beta},\bm{g})$ either regular or critical, we denote the closure of the inside of the loop by $P(t,\bm{\beta},\bm{g})$.

\begin{lemma}\label{replace2}
    Given a regular loop $(t_0,\bm{\beta},\bm{g})$ and its corresponding critical loop $(t_c,\bm{\beta},\bm{g})$, there exists a retraction \begin{equation*}
        r: P(t_0,\bm{\beta},\bm{g}) \to P(t_c,\bm{\beta},\bm{g}) \cup \bigcup_{\beta_{i,j}\in \bm{\beta}} \beta_{i,j}([t_0,t_c]).
    \end{equation*}
\end{lemma}

\begin{proof}
    This follows from the fact that for each $(t_0,\beta_{i,j},g_j,\beta_{j,k}) \subseteq (t_0,\bm{\beta},\bm{g})$, we can construct a retraction \begin{equation*}
        r_j: \overline{SQ(t_0,\beta_{i,j},g_j,\beta_{j,k})}  \to \beta_{i,j}([t_0,t_c]) \cup (t_c,\beta_{i,j},g_j,\beta_{j,k}) \cup \beta_{j,k}([t_0,t_c]).
    \end{equation*} We can glue $r_j$'s with the identity map on $P(t_c,\bm{\beta},\bm{g})$ using the pasting lemma.
\end{proof}

\begin{corollary}\label{replace}
    Suppose a measure $\mu$ and a compact set $K$ inside the polygon satisfy that \begin{enumerate}
        \item $U^{\mu}$ and $U^{\mu_E}$ coincide outside of the polygon,
        \item $\supp \mu \subseteq K$,
        \item $K$ is simply connected,
        \item there exists a regular loop $(t_0,\bm{\beta},\bm{g})$ inside $K$ such that $\mu|_{P(t_0,\bm{\beta},\bm{g})}$ is the loop measure of $(t_0,\bm{\beta},\bm{g})$,
        \item for such a regular loop, we have \begin{equation*}
            \overline{ K\backslash P(t_0,\bm{\beta},\bm{g})} \cap P(t_0,\bm{\beta},\bm{g}) \subseteq \set{\beta_{i,j}(t_0): \beta_{i,j}\in \bm{\beta}},
    \end{equation*}
    \end{enumerate}
    Then we can find a new measure $\mu'$ and a compact subset $K' \subseteq K$ strict smaller in area such that  \begin{enumerate}
        \item $\mu'$ and $K'$ coincide with $\mu$ and $K$ respectively on the exterior of $(t_0,\bm{\beta},\bm{g})$,
        \item $\mu'$ and $K'$ satisfy the first three conditions above,
        \item the critical loop $(t_c, \bm{\beta},\bm{g})$ is inside $K'$ and $\mu'|_{P(t_c,\bm{\beta},\bm{g})}$ is the loop measure of $(t_c, \bm{\beta},\bm{g})$,
        \item we have \begin{equation*}
            \overline{ K'\backslash P(t_c,\bm{\beta},\bm{g})} \cap P(t_c,\bm{\beta},\bm{g}) \subseteq \set{\beta_{i,j}(t_c): \beta_{i,j}\in \bm{\beta}},
    \end{equation*}
    \item $\mu'$ is equivalent to $\mathcal{H}^1$ on \begin{equation*}
        K' \cap (P(t_0,\bm{\beta},\bm{g}) \backslash  P(t_c,\bm{\beta},\bm{g})),
    \end{equation*}
    \end{enumerate} 
\end{corollary}

\begin{figure}[h!]
        \centering
        \begin{tikzpicture}[scale=0.8, line width=1.5pt]
            \begin{scope}
            \fill[gray, opacity=0.5] plot[smooth, tension=1] coordinates {(-1,-1) (0,-0.5) (1,-1)} -- plot[smooth, tension=1] coordinates {(1,-1) (1.2, -0.3) (2,0)} -- plot[smooth, tension=1] coordinates {(2,0) (1.1,0.6) (0.8,1.8)} -- plot[smooth, tension=1] coordinates {(0.8,1.8) (0.4,1.7) (0,2)} -- plot[smooth, tension=1] coordinates {(0,2) (-0.6,1) (-1.6,0.5)} -- plot[smooth, tension=1] coordinates {(-1.6,0.5) (-1.1,0) (-1,-1)} -- cycle;

            \fill[gray, opacity=0.5] plot[smooth, tension=1] coordinates {(-3,1) (-2.3,0.6) (-1.6,0.5)}
            -- plot[smooth, tension=1] coordinates {(-1.6,0.5) (-2.3,0.4)  (-3,0)} -- cycle;

            \draw[blue] plot[smooth, tension=1] coordinates {(-1,-1) (-0.8,-0.6) (-0.6, -0.3) (-0.5,-0.2)};
            \draw[blue] plot[smooth, tension=1] coordinates {(1,-1) (0.9,-0.7) (0.8,-0.4)};
            \draw[blue] plot[smooth, tension=1] coordinates {(2,0) (1.6,0.05)  (1.2,0)};
            \draw[blue] plot[smooth, tension=1] coordinates {(0.8,1.8) (0.6,1.6) (0.4,1.2) (0,1)};
            \draw[blue] plot[smooth, tension=1] coordinates {(0,2) (0.1,1.5) (0,1)};
            \draw[blue] plot[smooth, tension=1] coordinates {(-1.6,0.5) (-1.3,0.4) (-1,0.4) (-0.8,0.25)};

            \draw plot[smooth, tension=1] coordinates {(-2,-2) (-1.4,-1.6) (-1,-1)};
            \draw plot[smooth, tension=1] coordinates {(2,-2) (1.4,-1.6) (1,-1)};
            \draw plot[smooth, tension=1] coordinates {(3,0) (2.5,-0.1) (2,0) };
            \draw plot[smooth, tension=1] coordinates {(1.9,3) (1.45,2.3) (0.8,1.8)};
            \draw plot[smooth, tension=1] coordinates {(0,4) (0.2,3) (0,2)};
            \draw plot[smooth, tension=1] coordinates {(-2,3) (-0.9,2.6) (0,2)};
            \draw plot[smooth, tension=1] coordinates {(-3,1) (-2.3,0.6) (-1.6,0.5)};
            \draw plot[smooth, tension=1] coordinates {(-3,0) (-2.3,0.4) (-1.6,0.5)};

            \draw[red, line width=2pt] plot[smooth, tension=1] coordinates {(-1,-1) (0,-0.5) (1,-1)};
            \draw[red, line width=2pt]  plot[smooth, tension=1] coordinates {(1,-1) (1.2, -0.3) (2,0)}; 
            \draw[red, line width=2pt]  plot[smooth, tension=1] coordinates {(2,0) (1.1,0.6) (0.8,1.8)};
            \draw[red, line width=2pt]  plot[smooth, tension=1] coordinates {(0.8,1.8) (0.4,1.7) (0,2)};
            \draw[red, line width=2pt]  plot[smooth, tension=1] coordinates {(0,2) (-0.6,1) (-1.6,0.5)};
            \draw[red, line width=2pt]  plot[smooth, tension=1] coordinates {(-1.6,0.5) (-1.1,0) (-1,-1)};
            
            \draw[blue] plot[smooth, tension=1] coordinates {(-0.5,-0.2) (0.3,0.1)
            (0.8,-0.4)};
            \draw[blue] plot[smooth, tension=1] coordinates {(0.8,-0.4) (0.9,-0.2) (1.2,0)}; 
            \draw[blue] plot[smooth, tension=1] coordinates {(1.2,0) (0.3,0.15) (0,1)};
            \draw[blue] plot[smooth, tension=1] coordinates {(0,1) (-0.2,0.4) (-0.8,0.25)};
            \draw[blue] plot[smooth, tension=1] coordinates {(-0.8,0.25) (-0.6,0.1)(-0.5,-0.2)};
            \end{scope}

            \begin{scope}[shift={(8, 0)}]
            
            \fill[gray, opacity=0.5] plot[smooth, tension=1] coordinates {(-3,1) (-2.3,0.6) (-1.6,0.5)}
            -- plot[smooth, tension=1] coordinates {(-1.6,0.5) (-2.3,0.4)  (-3,0)} -- cycle;
            
            \draw plot[smooth, tension=1] coordinates {(-1,-1) (-0.8,-0.6) (-0.6, -0.3) (-0.5,-0.2)};
            \draw plot[smooth, tension=1] coordinates {(1,-1) (0.9,-0.7) (0.8,-0.4)};
            \draw plot[smooth, tension=1] coordinates {(2,0) (1.6,0.05)  (1.2,0)};
            \draw plot[smooth, tension=1] coordinates {(0.8,1.8) (0.6,1.6) (0.4,1.2) (0,1)};
            \draw plot[smooth, tension=1] coordinates {(0,2) (0.1,1.5) (0,1)};
            \draw plot[smooth, tension=1] coordinates {(-1.6,0.5) (-1.3,0.4) (-1,0.4) (-0.8,0.25)};

            \draw plot[smooth, tension=1] coordinates {(-2,-2) (-1.4,-1.6) (-1,-1)};
            \draw plot[smooth, tension=1] coordinates {(2,-2) (1.4,-1.6) (1,-1)};
            \draw plot[smooth, tension=1] coordinates {(3,0) (2.5,-0.1) (2,0) };
            \draw plot[smooth, tension=1] coordinates {(1.9,3) (1.45,2.3) (0.8,1.8)};
            \draw plot[smooth, tension=1] coordinates {(0,4) (0.2,3) (0,2)};
            \draw plot[smooth, tension=1] coordinates {(-2,3) (-0.9,2.6) (0,2)};
            \draw plot[smooth, tension=1] coordinates {(-3,1) (-2.3,0.6) (-1.6,0.5)};
            \draw plot[smooth, tension=1] coordinates {(-3,0) (-2.3,0.4) (-1.6,0.5)};

            \draw[dashed] plot[smooth, tension=1] coordinates {(-1,-1) (0,-0.5) (1,-1)};
            \draw[dashed] plot[smooth, tension=1] coordinates {(1,-1) (1.2, -0.3) (2,0)}; 
            \draw[dashed] plot[smooth, tension=1] coordinates {(2,0) (1.1,0.6) (0.8,1.8)};
            \draw[dashed] plot[smooth, tension=1] coordinates {(0.8,1.8) (0.4,1.7) (0,2)};
            \draw[dashed] plot[smooth, tension=1] coordinates {(0,2) (-0.6,1) (-1.6,0.5)};
            \draw[dashed] plot[smooth, tension=1] coordinates {(-1.6,0.5) (-1.1,0) (-1,-1)};
            
            \fill[gray, opacity=0.5] plot[smooth, tension=1] coordinates {(-0.5,-0.2) (0.3,0.1) (0.8,-0.4)} -- plot[smooth, tension=1] coordinates {(0.8,-0.4) (0.9,-0.2) (1.2,0)} -- plot[smooth, tension=1] coordinates {(1.2,0) (0.3,0.2) (0,1)} -- plot[smooth, tension=1] coordinates {(0,1) (-0.2,0.4) (-0.8,0.25)} -- plot[smooth, tension=1] coordinates {(-0.8,0.25) (-0.6,0.1)(-0.5,-0.2)} -- cycle;

            \draw[red,line width=2pt] plot[smooth, tension=1] coordinates {(-0.5,-0.2) (0.3,0.1) (0.8,-0.4)};
            \draw[red,line width=2pt] plot[smooth, tension=1] coordinates {(0.8,-0.4) (0.9,-0.2) (1.2,0)}; 
            \draw[red,line width=2pt] plot[smooth, tension=1] coordinates {(1.2,0) (0.3,0.2) (0,1)};
            \draw[red,line width=2pt] plot[smooth, tension=1] coordinates {(0,1) (-0.2,0.4) (-0.8,0.25)};
            \draw[red,line width=2pt] plot[smooth, tension=1] coordinates {(-0.8,0.25) (-0.6,0.1)(-0.5,-0.2)};
            \end{scope}
        \end{tikzpicture}
        \caption{Replace $(K,\mu)$ with $(K',\mu')$ where the loop measures on $(t_0,\bm{\beta},\bm{g})$ and  $(t_c,\bm{\beta},\bm{g})$ are highlighted in red}
        \label{fig:replace}
    \end{figure}
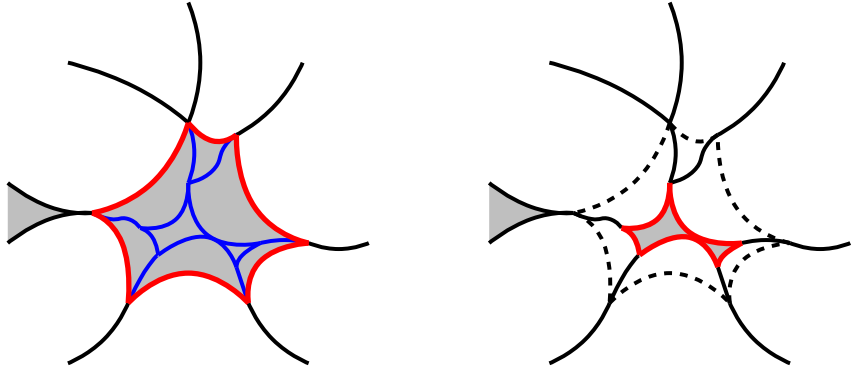

\begin{proof}
    The idea here is illustrated by Figure~\ref{fig:replace}. We replace $\mu$ with \begin{equation*}
        \mu'=\mu|_{P(t_0, \bm{\beta}, \bm{g})^C} + \mu_c+\lambda, \text{where $\mu_c$ and $\lambda$ are measures in Lemma~\ref{replace1}}, 
    \end{equation*} and $K$ with \begin{equation*}
        K'=\left(K\backslash P(t_0,\bm{\beta},\bm{g})\right) \cup P(t_c,\bm{\beta},\bm{g}) \cup \bigcup_{\beta_{i,j}\in \bm{\beta}} \beta_{i,j}([t_0,t_c]).
    \end{equation*} Then $\mu'$ generates the same logarithmic potential as $\mu$ outside $(t_0,\bm{\beta},\bm{g})$ due to Lemma~\ref{replace1}. According to Lemma~\ref{replace2}, we can construct a retraction from $K$ to $K'$, implying $K'$ is simply connected. Moreover, the strict descent condition implies that the density of $\lambda$, which is $|\nabla g_i - \nabla g_j|$, is strictly positive on $\beta_{i,j}([t_0,t_c])$. Thus $\lambda$ is equivalent to $\mathcal{H}^1$ on $\bigcup_{\beta_{i,j}\in \bm{\beta}} \beta_{i,j}([t_0,t_c])$.
\end{proof}

\subsection{Proof of Theorem~\ref{theo_descent}}
Given a convex polygon $\Omega$, we note that $\mu_E$ is the loop measure of \begin{equation*}
    \partial \Omega = (0,(\gamma_{1,2},\gamma_{2,3},\dots,\gamma_{n-1,n},\gamma_{1,n}),(g_2,g_3,\dots,g_n,g_1)).
\end{equation*}
Thus the pair $(\mu_E, \overline{\Omega})$ satisfies the five conditions in Corollary~\ref{replace}, implying that we can replace the pair with a new one $(\mu',K')$ which contains the loop measure of the critical loop \begin{equation}\label{eq67}
    (t_c,(\gamma_{1,2},\gamma_{2,3},\dots,\gamma_{n-1,n},\gamma_{1,n}),(g_2,g_3,\dots,g_n,g_1)).
\end{equation} Note that if this critical loop is not a point, then it is a union of regular loops intersecting each other only at the vertices per Proposition~\ref{breakdown}. Moreover, the restriction of the loop measure on each regular loop is, according to Definition~\ref{loop_measure}, the loop measure of the regular loop. Thus if the critical loop in $\eqref{eq67}$ is not a point, $(\mu',K')$ again satisfies the five conditions in Corollary~\ref{replace}. It follows that we can repeatedly apply Corollary~\ref{replace} to get a new pair $(\mu,K)$ where $K$ is a strict subset of the previous one. Note that after each replacement, $K$ is always a union of finitely many $P(t_0,\bm{\beta},\bm{g})$'s and $\beta_{i,j}([t_0,t_c])$'s.

According to Corollary~\ref{non-cross}, each time we replace a regular loop $(t_0,\bm{\beta},\bm{g})$ with its corresponding critical loop $(t_c,\bm{\beta},\bm{g})$, each regular loop inside $(t_c,\bm{\beta},\bm{g})$ (if any) corresponds to a sub-polygon of the polygon represented by $\bm{g}$. Since we cannot partition sub-polygons further than triangles, by Corollary~\ref{tri} the process terminates after finitely many steps where all critical loops we arrive at are degenerate. After the process terminates, $K$ will consist of solely analytic curves of the form $\beta_{i,j}([t_0,t_c])$. And the density of $\mu$ is strictly positive on each curve.

Now we bound the number of analytic curves on the resultant $K$. Note that when we rewrite a critical loop $(t_c,\bm{\beta},\bm{g})$ as a union of regular loops $\set{(t_c,\bm{\beta}_k,\bm{g}_k)}_{k=1}^K$, $\beta_{i,j}$ is in $\bm{\beta}_k$ for some $k$ but not in original $\bm{\beta}$ if and only if either \begin{enumerate}
    \item arcs corresponding to $g_i$ and $g_j$ in $(t_c,\bm{\beta},\bm{g})$ intersect in their interiors, or
    \item arcs corresponding to $g_i$ and $g_j$ in $(t_c,\bm{\beta},\bm{g})$ intersect at their end points (arcs in between become degenerate).
\end{enumerate} 
Each analytic curve in $K$ is either the image of one of $\gamma_{i,i+1}$'s or that of a new $\beta_{i,j}$ added in either case above. Moreover, in the first case, we are adding both a segment of $\gamma_{i,j}$ and of $\eta_{i,j}$ to $K$ but their images are connected at $m_{i,j}$ and thus make up a singular analytic curve. Therefore, when shrinking $(t_0,\bm{\beta},\bm{g})$ to $(t_c,\bm{\beta},\bm{g})$, the number of analytic curves we add into $K$ is upper bounded by the number of matching pairs in $\approx$ per Corollary~\ref{non-cross}. As a result, the total number of arcs added is bounded by the number of non-crossing matching pairs of a convex $n$-gon, which is $n-3$. Therefore, $K$ contains at most $2n-3$ analytic curves.

\bibliographystyle{plain}
\bibliography{bib}

\end{document}